\documentclass[12pt]{amsart}
\pdfoutput=1
\providecommand{\lang}{british}
\providecommand{\KOMAopt}{twoside=semi, DIV=default, headinclude}

\usepackage{eurosym}
\usepackage[\lang]{babel}
\usepackage[dvipsnames,svgnames]{xcolor}
\definecolor{darkblue}{rgb}{0.0,0.0,0.6}
\usepackage[utf8]{inputenc}
\usepackage[T1]{fontenc}
\usepackage{ae}
\usepackage{aecompl}
\usepackage{lmodern}
\usepackage{csquotes}
\usepackage[babel=true]{microtype}
\usepackage{scrextend}
\usepackage[\KOMAopt]{typearea}
\usepackage{amsmath}
\usepackage{mathtools}
	\mathtoolsset{mathic}
	
\usepackage{centernot}
\usepackage{amssymb}
\usepackage{amsfonts}
\usepackage{amsthm}
\usepackage{mathrsfs}
\usepackage{booktabs}
\usepackage{pdfpages}
\usepackage{bm}
\usepackage[only,llbracket,rrbracket]{stmaryrd}
\usepackage{upgreek}

\makeatletter
\@namedef{subjclassname@2020}{\textup{2020} Mathematics Subject Classification}
\makeatother

\usepackage[pdflang=en-UK, colorlinks, urlcolor=darkblue, linkcolor=darkblue, citecolor=darkblue]{hyperref}
\usepackage{color}
\usepackage{tikz}
	\usetikzlibrary{calc}
	\usetikzlibrary{arrows,decorations.markings}
\usepackage{tikz-cd}

\usepackage{todonotes}

\usepackage{enumitem}
\setenumerate{label=(\alph*)}

\usepackage{subcaption}

\newcommand{\Kdual}{\mathrm{D}}
\newcommand{\op}{\mathrm{op}}

\newcommand{\donothing}[1]{}

\newcommand{\dsum}{\mathbin{\oplus}}

\DeclareMathOperator{\Hom}{Hom}

\newcommand{\id}{\mathrm{id}}

\newcommand{\isoto}{\stackrel{\sim}{\to}}

\newcommand{\lperp}[1]{\prescript{\perp}{}{#1}}
	\newcommand{\rperp}[1]{{#1}^\perp}

\newcommand{\susp}{\Sigma}

\newcommand{\EE}{\mathbb{E}}

\newcommand{\KK}{\mathbb{K}}
\newcommand{\NN}{\mathbb{N}}

\newcommand{\RR}{\mathbb{R}}
\newcommand{\ZZ}{\mathbb{Z}}

\newcommand{\Lie}{\mathfrak}
\renewcommand{\sl}{\Lie{sl}}

\newcommand{\defeq}{\coloneqq}

	\DeclareMathOperator{\add}{add}

\DeclareMathOperator{\indec}{indec}

\DeclareMathOperator{\Modcat}{Mod}

\newcommand{\blank}{\mathord{\char123}}
\newcommand{\cat}{\mathcal}

\renewcommand{\epsilon}{\varepsilon}
\renewcommand{\geq}{\geqslant}
\renewcommand{\leq}{\leqslant}
\newcommand{\type}{\mathsf}

\newcommand{\bdry}{\partial}
\newcommand{\close}[1]{\overline{#1}}

\newcommand{\univ}{\tilde}  \theoremstyle{plain}
\newtheorem{thm}{Theorem}[section]
\newtheorem{thm*}{Theorem}
\newtheorem{lem}[thm]{Lemma}
\newtheorem{cor}[thm]{Corollary}
\newtheorem{cor*}{Corollary}
\newtheorem{prop}[thm]{Proposition}

\theoremstyle{definition}
\newtheorem{defn}[thm]{Definition}
\newtheorem{eg}[thm]{Example}

\theoremstyle{remark}
\newtheorem{rem}[thm]{Remark}

\numberwithin{equation}{section} \usepackage[backend=biber,citestyle=numeric-comp,bibstyle=numeric,maxbibnames=99,giveninits=true,doi=false,isbn=false,url=false,eprint=true]{biblatex}

\renewbibmacro{in:}{\ifentrytype{article}{}{\printtext{\bibstring{in}\intitlepunct}}}

\DeclareFieldFormat[article,inbook,incollection,inproceedings,patent]{title}{#1}
\DeclareFieldFormat[thesis,unpublished]{title}{{\it #1}}

\DeclareFieldFormat[online]{date}{Preprint (#1)}

\DeclareFieldFormat[article]{volume}{\mkbibbold{#1}}
\renewbibmacro*{volume+number+eid}{\printfield{volume}\setunit{\addcomma\space}\printfield{eid}}

\DeclareFieldFormat[book,incollection]{number}{Vol.~#1}
\renewbibmacro*{series+number}{\iffieldundef{series}{}
    {\printfield{series}\iffieldundef{number}{}{\setunit*{\addcomma\addspace}\printfield{number}\newunit}}}

\renewbibmacro*{publisher+location+date}{\printlist{publisher}\setunit*{\addcomma\space}\usebibmacro{date}\newunit}

\DeclareFieldFormat{eprint:arxiv}{\ifentrytype{online}
  {\ifhyperref
    {\href{http://arxiv.org/abs/#1}{\nolinkurl{arXiv:#1}}}
    {\nolinkurl{arXiv:#1}}
   \iffieldundef{eprintclass}
    {}
    {{\texttt{[\thefield{eprintclass}]}}}}
  {\iffieldundef{eprintclass}
    {\mkbibparens{\ifhyperref
    {\href{http://arxiv.org/abs/#1}{\nolinkurl{arXiv:#1}}}
    {\nolinkurl{arXiv:#1}}}}
    {\mkbibparens{\ifhyperref
    {\href{http://arxiv.org/abs/#1}{\nolinkurl{arXiv:#1}}}
    {\nolinkurl{arXiv:#1}}
    {\texttt{[\thefield{eprintclass}]}}}}}}

\usepackage{xpatch}

\makeatletter
\patchcmd\blx@bblinput{\blx@blxinit}
                      {\blx@blxinit
                      }{}{\fail}
\makeatother

\title[Cluster categories for completed infinity-gons I]{Cluster categories for completed infinity-gons I: Categorifying triangulations}
\author{İlke Çanakçı}
\address[İ.\ Çanakçı]{Department of Mathematics\\Vrije Universiteit Amsterdam\\De Boelelaan 1111\\1081~HV~Amsterdam\\Netherlands}
\email{\href{mailto:i.canakci@vu.nl}{i.canakci@vu.nl}}
\urladdr{\url{https://sites.google.com/view/ilkecanakci/}}
\author{Martin Kalck}
\address[M.\ Kalck]{Institut für Mathematik und Wissenschaftliches Rechnen\\Universität Graz\\Heinrichstraße 36\\8010 Graz\\Austria}
\email{\href{mailto:martin.kalck@uni-graz.at}{martin.kalck@uni-graz.at}}
\urladdr{\url{https://imsc.uni-graz.at/kalck/}}
\author{Matthew Pressland}
\address[M.\ Pressland]{School of Mathematics \& Statistics\\University of Glasgow\\University Place\\Glasgow G12~8QQ\\United Kingdom}
\email{\href{mailto:matthew.pressland@unicaen.fr}{matthew.pressland@unicaen.fr}}
\urladdr{\url{https://mdpressland.github.io}}

\newcommand{\disc}{\mathbb{D}}
\newcommand{\HomT}{\mathbf{F}}

\newcommand{\ITcat}{\mathcal{C}}
\renewcommand{\lim}{\operatorname{lim}}

\newcommand{\loc}{\mathbf{V}}

\newcommand{\pts}{\mathbb{M}}
\newcommand{\PYcat}{\overline{\mathcal{C}}}
\newcommand{\s}{\mathfrak{s}}
\newcommand{\Serre}{\mathbf{S}}

\newcommand{\squeeze}[1]{\overline{#1}}

\newcommand{\tri}{\mathbb{T}}
\renewcommand{\univ}{\widetilde}

\renewcommand{\subset}{\subseteq}

\newcommand{\infl}{\rightarrowtail}
\newcommand{\defl}{\twoheadrightarrow}
\newcommand{\confl}{\dashrightarrow}
\tikzcdset{
infl/.style={rightarrowtail},
defl/.style={twoheadrightarrow},
confl/.style={dashed},
exists/.style={dotted}}

\tikzset{
match/.style={line width=2.2pt}}

\makeatletter

\def\hyper@x#1,#2\relax{#1}
\def\hyper@y#1,#2\relax{#2}
\def\hyper@coords#1{#1}

\newif\ifhyper@vertical

\def\hyper@computer#1#2{\edef\hyper@toscan{(#1)}
  \tikz@scan@one@point\hyper@coords\hyper@toscan
  \edef\hyper@sx{\the\pgf@x}
  \edef\hyper@sy{\the\pgf@y}
  \edef\hyper@toscan{(#2)}
  \tikz@scan@one@point\hyper@coords\hyper@toscan
  \edef\hyper@ex{\the\pgf@x}
  \edef\hyper@ey{\the\pgf@y}
  \pgfmathsetmacro{\hyper@mx}{(\hyper@ex + \hyper@sx)/2}
  \pgfmathsetmacro{\hyper@my}{(\hyper@ey + \hyper@sy)/2}
  \pgfmathsetmacro{\hyper@dx}{\hyper@ex - \hyper@sx}
  \pgfmathparse{\hyper@dx == 0 ? "\noexpand\hyper@verticaltrue" : "\noexpand\hyper@verticalfalse"}
  \pgfmathresult
  \ifhyper@vertical
  \edef\hyper@cmd{-- (\tikztotarget)}
  \else
  \pgfmathsetmacro{\hyper@dy}{\hyper@ey - \hyper@sy}
  \pgfmathsetmacro{\hyper@t}{\hyper@my/\hyper@dx}
  \pgfmathsetmacro{\hyper@cx}{\hyper@mx + \hyper@t * \hyper@dy}
  \pgfmathsetmacro{\hyper@radius}{veclen(\hyper@cx - \hyper@sx, \hyper@sy)}
  \pgfmathsetmacro{\hyper@sangle}{180 - atan2(\hyper@sy,\hyper@cx-\hyper@sx)}
  \pgfmathsetmacro{\hyper@eangle}{180 - atan2(\hyper@ey,\hyper@cx-\hyper@ex)}
  \edef\hyper@cmd{arc[radius=\hyper@radius pt, start angle=\hyper@sangle, end angle=\hyper@eangle]}
  \fi
}

\def\hyper@disc@computer#1#2{\edef\hyper@toscan{(#1)}
  \tikz@scan@one@point\hyper@coords\hyper@toscan
  \edef\hyper@sx{\the\pgf@x}
  \edef\hyper@sy{\the\pgf@y}
  \edef\hyper@toscan{(#2)}
  \tikz@scan@one@point\hyper@coords\hyper@toscan
  \edef\hyper@ex{\the\pgf@x}
  \edef\hyper@ey{\the\pgf@y}
  \pgfmathsetmacro{\hyper@det}{\hyper@sx * \hyper@ey - \hyper@sy * \hyper@ex}
  \pgfmathparse{\hyper@det == 0 ? "\noexpand\hyper@verticaltrue" : "\noexpand\hyper@verticalfalse"}
  \pgfmathresult
  \ifhyper@vertical
  \edef\hyper@cmd{-- (\tikztotarget)}
  \else
  \pgfmathsetmacro{\hyper@mx}{(\hyper@ex + \hyper@sx)/2}
  \pgfmathsetmacro{\hyper@my}{(\hyper@ey + \hyper@sy)/2}
  \pgfmathsetmacro{\hyper@dx}{\hyper@ex - \hyper@sx}
  \pgfmathsetmacro{\hyper@dy}{\hyper@ey - \hyper@sy}
  \pgfmathsetmacro{\hyper@dradius}{\pgfkeysvalueof{/tikz/hyperbolic disc radius}}
  \pgfmathsetmacro{\hyper@t}{((\hyper@dradius)^2 - \hyper@sx * \hyper@ex - \hyper@sy * \hyper@ey)/(2 * (\hyper@sx * \hyper@ey - \hyper@sy * \hyper@ex))}
  \pgfmathsetmacro{\hyper@radius}{sqrt((\hyper@t)^2 + .25) * veclen(\hyper@dx,\hyper@dy)}
  \pgfmathsetmacro{\hyper@cx}{\hyper@mx + \hyper@t * \hyper@dy}
  \pgfmathsetmacro{\hyper@cy}{\hyper@my - \hyper@t * \hyper@dx}
  \pgfmathsetmacro{\hyper@sangle}{atan2(\hyper@sy-\hyper@cy,\hyper@sx - \hyper@cx)}
  \pgfmathsetmacro{\hyper@eangle}{atan2(\hyper@ey-\hyper@cy,\hyper@ex - \hyper@cx)}
  \pgfmathsetmacro{\hyper@eangle}{\hyper@eangle > \hyper@sangle + 180 ? \hyper@eangle - 360 : \hyper@eangle}
  \edef\hyper@cmd{arc[radius=\hyper@radius pt, start angle=\hyper@sangle, end angle=\hyper@eangle]}
\fi
}

\def\hyper@plane@tangent#1#2{\edef\hyper@toscan{(#1)}
  \tikz@scan@one@point\hyper@coords\hyper@toscan
  \edef\hyper@sx{\the\pgf@x}
  \edef\hyper@sy{\the\pgf@y}
  \edef\hyper@toscan{(#2)}
  \tikz@scan@one@point\hyper@coords\hyper@toscan
  \edef\hyper@ex{\the\pgf@x}
  \edef\hyper@ey{\the\pgf@y}
\pgfmathsetmacro{\hyper@ex}{\hyper@ex - \hyper@sx}
  \pgfmathsetmacro{\hyper@ey}{\hyper@ey - \hyper@sy}
\pgfmathparse{\hyper@ex == 0 ? "\noexpand\hyper@verticaltrue" : "\noexpand\hyper@verticalfalse"}
  \pgfmathresult
  \ifhyper@vertical
\pgfmathsetmacro{\hyper@d}{\hyper@ey/1cm}
  \pgfmathsetmacro{\hyper@radius}{\hyper@sy * exp(\hyper@d) - \hyper@sy}
  \edef\hyper@cmd{-- ++(0,\hyper@radius pt)}
  \else
\pgfmathsetmacro{\hyper@d}{\hyper@ex > 0 ? veclen(\hyper@ex,\hyper@ey) : -veclen(\hyper@ex,\hyper@ey)}
\pgfmathsetmacro{\hyper@radius}{abs(\hyper@sy * \hyper@d / \hyper@ex)}
\pgfmathsetmacro{\hyper@sangle}{90 + atan(\hyper@ey/\hyper@ex)}
\pgfkeysgetvalue{/tikz/hyperbolic plane target angle}{\hyper@eangle}
  \ifx\hyper@eangle\pgfutil@empty
\pgfmathsetmacro{\hyper@d}{\hyper@d/1cm}
  \pgfmathsetmacro{\hyper@ey}{\hyper@ey/1cm}
  \pgfmathsetmacro{\hyper@tanhd}{tanh(\hyper@d)}
  \pgfmathsetmacro{\hyper@eangle}{acos((\hyper@d * \hyper@tanhd - \hyper@ey)/(\hyper@d - \hyper@ey * \hyper@tanhd))}
\fi
  \edef\hyper@cmd{arc[radius=\hyper@radius pt, start angle=\hyper@sangle, end angle=\hyper@eangle]}
\fi
}

\tikzset{hyperbolic disc radius/.initial={1cm},
  hyperbolic plane/.style={
    to path={
      \pgfextra{\hyper@computer\tikztostart\tikztotarget}
      \hyper@cmd
    }
  },
  hyperbolic plane tangent/.style={
    to path={
      \pgfextra{\hyper@plane@tangent\tikztostart\tikztotarget}
      \hyper@cmd
    }
  },
  hyperbolic disc/.style={
    to path={
      \pgfextra{\hyper@disc@computer\tikztostart\tikztotarget}
      \hyper@cmd
    }
  },
  hyperbolic plane target angle/.initial={},
}

\makeatother

\begin{document}

\begin{abstract}
Paquette and Yıldırım recently introduced triangulated categories of arcs in completed infinity-gons, which are discs with an infinite closed set of marked points on their boundary.
These categories have many features in common with the cluster categories associated to discs with different sets of marked points.
In particular, they have (weak) cluster-tilting subcategories, which Paquette--Yıldırım show are in bijection with very special triangulations of the disc.
This is in contrast to Igusa--Todorov's earlier work in the uncompleted case, in which every triangulation corresponds to a weak cluster-tilting subcategory.

In this paper, we replace the triangulated structure of Paquette--Yıldırım's category by an extriangulated substructure and prove that, with this structure, the weak cluster-tilting subcategories are once again in bijection with triangulations.
We further show that functorial finiteness of a weak cluster-tilting subcategory is equivalent to a very mild condition on the triangulation, which also appears in Çanakçı and Felikson's study of infinite rank cluster algebras from Teichmüller theory.
By comparison with the combinatorics of triangulations, we are also able to characterise when weak cluster-tilting subcategories can be mutated in this new extriangulated category.
\end{abstract}
\maketitle

\section{Introduction}

Paquette and Yıldırım \cite{PaqYil} introduced a triangulated category \(\PYcat_n\) associated to a disc \(\disc\) with a closed set \(\close{\pts}_n\) of infinitely many marked points in \(\bdry \disc\), such that \(\close{\pts}_n\) has \(n\) accumulation points, all of which are two-sided.
As a special case, the category \(\PYcat_1\) has also appeared in various other contexts and guises; see for example \cite{ACFGS1,ACFGS2,Fisher,Kalck-Thesis}.
The indecomposable objects of \(\PYcat_n\) are in one-to-one correspondence with arcs in \(\disc\) connecting points of \(\close{\pts}_n\), allowing Paquette--Yıldırım to classify the cluster-tilting subcategories of \(\PYcat_n\) in terms of triangulations.
However, only very few triangulations turn out to correspond to cluster-tilting subcategories (or even to weak cluster-tilting subcategories, which are not required to be functorially finite, cf.~\cite{GHJ}).
Indeed, the failure of \(\PYcat_n\) to be \(2\)-Calabi--Yau means that a triangulation corresponding to a cluster-tilting subcategory may contain at most one arc incident to each accumulation point, and no arcs incident to two accumulation points.

Thus the cluster-tilting theory of \(\PYcat_n\) appears to describe only a very small part of the combinatorics of triangulations of \(\disc\), in contrast to the categories associated by Igusa and Todorov to surfaces with discrete sets of marked points \cite{IguTod-CyclPoset}.
This stands in contrast with results in Teichmüller theory \cite{Penner-Book}, to which surface cluster algebras have a close relationship \cite{FST1, GSV-WeilPet, FocGon-HTT}.
In particular, triangulations of a finite marked surface define both a cluster in the associated surface cluster algebra and a coordinate system for the associated (decorated) Teichmüller space.
An analogous result for the case of infinitely many marked points is given by Çanakçı--Felikson \cite[Cor.~3.35]{CanFel}.
For analytic reasons, this result applies only to fan triangulations (see Definition~\ref{d:leapfrog}), but this class is much larger than that of triangulations corresponding to cluster-tilting subcategories in \(\PYcat_n\), and includes for example the triangulation shown in Figure~\ref{f:triang-non-PY-ct}.

\begin{figure}
\begin{tikzpicture}[scale=2.4,every to/.style={hyperbolic disc}]
\draw (270:1) to (90:1);
\foreach \t in {60,80,100,120,140}
{\draw (270:1) to (90+\t:1);
\draw (90-\t:1) to (90:1);}
\draw (270:1) to (130:1);
\draw (50:1) to (90:1);
\foreach \t in {-2,0,2}
{\draw (110+\t:0.95) node {\({\cdot}\)};
\draw (70+\t:0.95) node {\({\cdot}\)};
\draw (250+\t:0.95) node {\({\cdot}\)};
\draw (290+\t:0.95) node {\({\cdot}\)};}
\draw[gray,very thick] (0,0) circle(1);
\foreach \n in {1,...,2}
{ \draw (90-180*\n:1) circle(0.025) [fill=black];}
\end{tikzpicture}
\caption{A fan triangulation \(\tri\) of \((\disc,\close{\pts}_2)\) (with accumulation points shown as filled circles).
This triangulation does not correspond to a weak cluster-tilting subcategory in \(\PYcat_n\) with its original triangulated structure, due to the presence of the arc between two accumulation points, and multiple arcs incident with each accumulation point, but it does correspond to a cluster-tilting subcategory of \(\PYcat_n\) with the appropriate extriangulated structure.}
\label{f:triang-non-PY-ct}
\end{figure}

In this paper, we demonstrate that in fact \(\PYcat_n\) does encode the combinatorics of all triangulations of the marked surface \((\disc,\close{\pts}_n)\), and surprisingly even does so through cluster-tilting theory.
The key observation is that this is achieved not with the triangulated structure on \(\PYcat_n\), but by passing to an extriangulated substructure in the sense of Nakaoka and Palu \cite{NakPal}.
Roughly speaking, this corresponds to removing certain extensions from \(\PYcat_n\), while keeping the underlying additive category the same, thus allowing more objects and subcategories to become rigid.
\begin{thm*}[Corollary~\ref{c:wct}]
\label{t:wct}
The map \(\cat{T}\mapsto\indec{\cat{T}}\) is a bijection between weak cluster-tilting subcategories of \(\PYcat_n\), with the appropriate extriangulated structure, and triangulations of the disc \((\disc,\close{\pts}_n)\).
\end{thm*}

Intriguingly, the definition of this extriangulated structure on \(\PYcat_n\) arises naturally from its description as the Verdier quotient of a different triangulated category, first introduced by Igusa and Todorov \cite{IguTod-CyclPoset}.
Indeed, we show in Section~\ref{s:substruct} how to produce two extriangulated substructures, called \(\EE_{\lperp{\cat{D}}}\) and \(\EE_{\rperp{\cat{D}}}\), on the Verdier quotient \(\overline{\cat{C}}=\cat{C}/\cat{D}\) of a triangulated category \(\cat{C}\) by a thick subcategory \(\cat{D}\)---these two substructures coincide when \(\cat{C}\) is Calabi--Yau.
The resulting extriangulated structure depends on the input data \(\cat{D}\subseteq\cat{C}\) for the Verdier quotient construction, and not just the quotient category \(\overline{\cat{C}}\) itself up to triangle equivalence, and so in particular it is usually different from Verdier's triangulated structure \cite{Verdier-These} on \(\overline{\cat{C}}\).

Having recovered the dictionary between triangulations and weak cluster-tilting subcategories (Section~\ref{s:ct-subcats}), we continue by translating further phenomena between the two languages.
In Section~\ref{s:approx} we determine the triangulations for which the corresponding subcategory is cluster-tilting (i.e.\ functorially finite).
These turn out to be precisely the fan triangulations, used by Çanakçı--Felikson \cite{CanFel} to give coordinates on the associated decorated Teichmüller space.

\begin{thm*}[Theorem~\ref{t:ct=fan}]
The map from Theorem~\ref{t:wct} restricts to a bijection between cluster-tilting subcategories of \(\PYcat_n\), with the appropriate extriangulated structure, and fan triangulations of the disc \((\disc,\close{\pts}_n)\).
\end{thm*}

Finally, in Section~\ref{s:mutation}, we show that our bijection between weak cluster-tilting subcategories and triangulations is compatible with mutation.
In contrast to surfaces with finitely many marked points, an arc may be mutable (i.e.\ flippable) in one triangulation but not in another, and we explain how this feature is reflected categorically in terms of the existence of certain approximations.
\begin{thm*}[Theorem~\ref{t:mutability}]
\label{t*:mutability}
Let \(\cat{T}\) be a cluster-tilting subcategory of \(\PYcat_n\) (with the appropriate extriangulated structure), and let \(\alpha\in\indec{\cat{T}}\) be an arc. Then the following are equivalent:
\begin{enumerate}
\item \(\cat{T}\) is mutable at \(\alpha\);
\item \(\alpha\) is the diagonal of a quadrilateral in the corresponding triangulation;
\item\label{t*:mutability-approx} \(\alpha\) admits left and right approximations by \(\cat{T}\setminus\{\alpha\}=\add(\indec(\cat{T})\setminus\{\alpha\})\).
\end{enumerate}
\end{thm*}

The failure of an arc \(\alpha\) to be mutable in a given triangulation \(\tri\) is meaningful in Teichmüller theory, since this happens precisely when the lambda length of the arc is the limit of a sequence of lambda lengths of other arcs in \(\tri\) \cite[\S3.2]{CanFel}; that is, when \(\alpha\) is a limit arc in \(\tri\) \cite[Defn.~2.5]{CanFel}. This is reflected categorically in Theorem~\ref{t*:mutability}\ref{t*:mutability-approx}, which implies that any attempt to approximate a non-mutable arc \(\alpha\) by arcs from \(\cat{T}\setminus\{\alpha\}\) may always be improved, leading to such an infinite sequence.
The notion of limit arc here, which depends on the triangulation \(\tri\), is different from that of Paquette--Yıldırım, who use the term to refer to any arc incident with an accumulation point---such arcs can be mutable or not, depending on the rest of the triangulation.

We will further explore connections to Teichmüller theory and Çanakçı--Felikson's infinite rank surface cluster algebras (for discs) in a sequel \cite{CKaP2} to this article, in which we will use our extriangulated structure to describe a cluster character on \(\PYcat_n\) which computes formal lambda lengths of arcs.

\section*{Acknowledgements}
This project began during the Junior Trimester Programme \emph{New Trends in Representation Theory} at the Hausdorff Institute for Mathematics in Bonn, Germany, in autumn 2020.
We thank Gustavo Jasso and Jan Schröer for organising the programme, and the Hausdorff Institute for financial support and an inspiring environment, especially under the difficult circumstances at the time.
We further thank Jenny August, Sofia Franchini, Marina Godinho, Mikhail Gorsky, Sira Gratz, Martin Herschend, Gustavo Jasso, Carlo Klapproth, Dave Murphy, Yann Palu, Charles Paquette, Pierre-Guy Plamondon, Amit Shah and Emine Yıldırım for useful conversations at various stages of this project, and the anonymous referee for their helpful comments and corrections.

M.K.\ was partially funded by the Deutsche Forschungsgemeinschaft (DFG, German Research Foundation), Projektnummern 496500943; 201167725.
He is also grateful to Michael Wemyss for inviting him for a research stay in Glasgow and to Peter Jørgensen for an invitation to Aarhus, where parts of this work were done.
M.P.\ was supported by the EPSRC Postdoctoral Fellowship EP/T001771/2, and ERC Consolidator Grant 101001227 (MMiMMa).
He also thanks VU Amsterdam and Universität Graz for funding research visits in 2023. Parts of this work were done at the \emph{Cluster algebras and representation theory} programme in 2021 at the Isaac Newton Institute for Mathematical Sciences (supported by EPSRC grant no EP/R014604/1).

\section{Extriangulated structures on quotients}
\label{s:substruct}

Let \(\cat{C}\) be a triangulated category with thick subcategory \(\cat{D}\).
Then one can form the Verdier quotient \(\cat{C}/\cat{D}\), which is again a triangulated category.
In this section we will give various recipes for equipping \(\cat{C}/\cat{D}\) with the structure of an extriangulated category, in the sense of Nakaoka and Palu \cite{NakPal}, by discarding some of its distinguished triangles.

We begin by recalling the fundamental notions related to Verdier quotients of triangulated categories.

\begin{defn}[{\cite[\S2]{Verdier-These}, see also \cite[Thm.~2.1.8]{Neeman-TriCats}}]
Let \(\cat{C}\) be a triangulated category and let \(\cat{D}\subset\cat{C}\) be a thick subcategory.
A morphism \(f\in\cat{C}\) is a \emph{\(\cat{D}\)-quasi-isomorphism} if its mapping cone lies in \(\cat{D}\), and the \emph{Verdier quotient} (or \emph{Verdier localisation}) \(\cat{C}/\cat{D}\) of \(\cat{C}\) by \(\cat{D}\) is the localisation of \(\cat{C}\) in the class of \(\cat{D}\)-quasi-isomorphisms.
\end{defn}

The quotient \(\cat{C}/\cat{D}\) naturally inherits the structure of a triangulated category from \(\cat{C}\), and there is a canonical triangle functor \(\loc\colon\cat{C}\to\cat{C}/\cat{D}\).
The map on objects underlying \(\loc\) is the identity, and so we will often suppress \(\loc\) in the notation when it is applied to an object.
However, \(\loc\) is typically neither full nor faithful, and so we will emphasise this functor when it is applied to a morphism.
Indeed, an important property for us is that \(\ker{\loc}\) is precisely the ideal \((\cat{D})\) consisting of morphisms of \(\cat{C}\) which factor over an object of \(\cat{D}\) \cite[Prop.~4.6.2(2)]{Krause-Loc}.

\begin{defn}
Let \(\cat{A}\) be an additive category and \(\cat{B}\) a subcategory.
We write
\begin{align*}
\lperp{\cat{B}}&=\{X\in\cat{A}:\text{\(\Hom_{\cat{A}}(X,B)=0\) for all \(B\in\cat{B}\)}\},\\
\rperp{\cat{B}}&=\{X\in\cat{A}:\text{\(\Hom_{\cat{A}}(B,X)=0\) for all \(B\in\cat{B}\)}\},
\end{align*}
and call these the \emph{left}, respectively \emph{right}, \emph{perpendicular subcategories} to \(\cat{B}\).
\end{defn}

The next proposition follows from the description of morphisms in \(\cat{C}/\cat{D}\) via a calculus of fractions.

\begin{prop}[{\cite[Prop.~2.3.3]{Verdier-These}}]
\label{p:Dperp}
Let \(\cat{C}\) be a triangulated category with thick subcategory \(\cat{D}\) and let \(X,Y\in\cat{C}\).
Then \(\loc\colon \Hom_{\cat{C}}(X,Y)\to \Hom_{\cat{C}/\cat{D}}(X,Y)\) is an isomorphism whenever \(X\in\lperp{\cat{D}}\) or \(Y\in\rperp{\cat{D}}\).
\end{prop}

\begin{cor}
\label{c:perp-equiv}
The Verdier quotient functor \(\loc\colon\cat{C}\to\cat{C}/\cat{D}\) restricts to an equivalence from \(\lperp{\cat{D}}\) to its essential image in \(\cat{C}/\cat{D}\), and similarly for \(\rperp{\cat{D}}\).
\end{cor}

Because of Corollary~\ref{c:perp-equiv} (and the fact that \(\cat{C}\) and \(\cat{C}/\cat{D}\) have the same objects), we typically abuse notation and write \(\lperp{\cat{D}}\) and \(\rperp{\cat{D}}\) also for their essential images under \(\loc\).
Note that when \(\cat{C}\) admits a Serre functor \(\Serre\) which preserves \(\cat{D}\), we have \(\lperp{\cat{D}}=\rperp{\cat{D}}\) since
\[\Hom_{\cat{C}}(D,X)=\Kdual{\Hom_{\cat{C}}(X,\Serre D)}\]
for all \(X\in\cat{C}\) and \(D\in\cat{D}\).
This holds in particular if \(\cat{C}\) is \(n\)-Calabi--Yau, meaning that \(\Serre=\susp^n\), because \(\cat{D}\) is a triangulated subcategory of \(\cat{C}\) and hence preserved by \(\susp\).

The general theory of extriangulated categories is technical, but because all of our examples will be obtained from triangulated categories by taking a substructure, we can give a simplified definition that deals only with this situation.

\begin{defn}
\label{d:substr}
Let \(\cat{C}\) be a triangulated category with suspension \(\susp\), and let \(\EE\leq\Hom_{\cat{C}}(\blank,\susp\blank)\) be a subfunctor.
For each triangle
\[\begin{tikzcd}
X\arrow{r}{f}&Y\arrow{r}{g}&Z\arrow{r}{\delta}&\susp X
\end{tikzcd}\]
of \(\cat{C}\) with \(\delta\in\EE(Z,X)\), we call the map \(f\) an \emph{\(\EE\)-inflation}, the map \(g\) an \emph{\(\EE\)-deflation}, and the triple \((f,g,\delta)\) an \emph{\(\EE\)-conflation}.
When \(\EE\) is implicit, we will refer simply to inflations, deflations and conflations.
We say that \(\EE\) is an \emph{extriangulated substructure} for \(\cat{C}\) if the sets of \(\EE\)-inflations and \(\EE\)-deflations are each closed under composition.
The notation
\[\begin{tikzcd}
X\arrow[infl]{r}{f}&Y\arrow[defl]{r}{g}&Z\arrow[confl]{r}{\delta}&\phantom{}
\end{tikzcd}\]
indicates that \(f\) is an inflation, \(g\) is a deflation, and \((f,g,\delta)\) is a conflation.
\end{defn}

\begin{rem}
\label{r:HLN}
A result of Herschend, Liu and Nakaoka \cite[Prop.~3.16]{HLN1} (specialised to \(n=1\)) states that \(\EE\)-inflations are closed under composition if and only if \(\EE\)-deflations are closed under composition, so it suffices to check one of the two conditions in practice.
Note that, unlike the set of distinguished triangles, the set of \(\EE\)-conflations need not be closed under rotation.
Moreover, inflations need not be monomorphisms, and deflations need not be epimorphisms.

The same result of \cite{HLN1} implies that extriangulated substructures as in Definition~\ref{d:substr} coincide with those in the full, technical sense of Nakaoka and Palu: that is, all extriangulated structures \((\EE,\mathfrak{s})\) on \(\cat{C}\) such that \(\EE\leq\Hom_{\cat{C}}(\blank,\susp\blank)\), and \(\mathfrak{s}\) is restricted from the standard realisation map completing morphisms to triangles, are obtained in this way.
\end{rem}

The following construction will be our main tool for producing extriangulated substructures on triangulated categories.
It is a special case of \cite[Eg.~3.5(3)]{BalSte} (see \cite[Appendix]{Sakai} for the translation to the language of extriangulated substructures), but we give the proof for convenience.

Let \(F\colon\cat{C}\to\cat{D}\) be a triangle functor.
For each \(X,Z\in\cat{C}\), define \(\EE_F(Z,X)\leq\Hom(Z,\susp X)\) to be the set of morphisms \(\delta\colon Z\to\susp X\) for which \(F\delta=0\).
Recall that every monomorphism or epimorphism in a triangulated category is split.

\begin{lem}
\label{l:functor-construction}
The functor \(\EE_F\) defines an extriangulated substructure on \(\cat{C}\).
A morphism \(f\colon X\to Y\) in \(\cat{C}\) is an \(\EE_F\)-inflation if and only if \(Ff\) is a monomorphism in \(\cat{D}\). Dually, a morphism \(g\colon Y\to Z\) is an \(\EE_F\)-deflation if and only if \(Fg\) is an epimorphism in \(\cat{D}\).
\end{lem}
\begin{proof}
Observe that \(\EE_F(\blank,\blank)\) is a subfunctor of \(\Hom(\blank,\susp\blank)\) because \(\ker{F}\) is an ideal of \(\cat{C}\).
Let
\[\begin{tikzcd}
X\arrow{r}{f}&Y\arrow{r}{g}&Z\arrow{r}{\delta}&\susp X
\end{tikzcd}\]
be a triangle of \(\cat{C}\).
Since \(F\) is a triangle functor, 
\[\begin{tikzcd}
FX\arrow{r}{Ff}&FY\arrow{r}{Fg}&FZ\arrow{r}{F\delta}&F\susp X
\end{tikzcd}\]
is a triangle of \(\cat{D}\).
Thus, \(Ff\) is a monomorphism, and \(Fg\) an epimorphism, if and only if \(F\delta=0\), i.e.\ if and only if \(\delta\in\EE_F(Z,X)\).
This demonstrates that the \(\EE_F\)-inflations and \(\EE_F\)-deflations are exactly as claimed.
These are closed under composition since this is true of monomorphisms and epimorphisms, and hence \(\EE_F\) is an extriangulated substructure for \(\cat{C}\).
\end{proof}

A much more general version of Lemma~\ref{l:functor-construction}, concerning functors between arbitrary extriangulated categories, is provided in Appendix~\ref{appendix}.
This general statement (Proposition~\ref{p:extri-pullback}) implies that \(F\) determines an extriangulated functor \cite[Def.~3.15]{BTHSS} from \(\cat{C}\) with the substructure \(\EE_F\) to \(\cat{D}\) with the split exact structure, and that \(\EE_F\) is the maximal extriangulated substructure on \(\cat{C}\) having this property.
However, the simplified statement and proof of Lemma~\ref{l:functor-construction} will be sufficient for our main applications in this paper.

\begin{defn}
\label{d:subcat-substr}
Given a triangulated category \(\cat{C}\) with thick subcategory \(\cat{D}\), define
\[\EE_\cat{D}(Z,X)=(\cat{D})(Z,\susp X).\]
That is, \(\EE_{\cat{D}}(Z,X)\) is the subspace of \(\Hom_{\cat{D}}(Z,\susp X)\) consisting of morphisms factoring over the subcategory \(\cat{D}\).
\end{defn}

\begin{cor}
\label{c:subcat-substr}
Let \(\cat{C}\) be a triangulated category with thick subcategory \(\cat{D}\).
Then \(\EE_{\cat{D}}\) is an extriangulated substructure of \(\cat{C}\).
\end{cor}
\begin{proof}
This follows directly from Lemma~\ref{l:functor-construction} applied to the Verdier quotient functor \(\loc\colon\cat{C}\to\cat{C}/\cat{D}\), observing that \(\EE_{\cat{D}}=\EE_{\loc}\) since \(\ker{\loc}=(\cat{D})\) \cite[Prop.~4.6.2(2)]{Krause-Loc}.
\end{proof}

In practice, we will use Corollary~\ref{c:subcat-substr} as follows.
Given a triangulated category \(\cat{C}\) with thick subcategory \(\cat{D}\), the Verdier quotient \(\overline{\cat{C}}=\cat{C}/\cat{D}\) has thick subcategories \(\lperp{\cat{D}}\) and \(\rperp{\cat{D}}\) (viewed as subcategories of \(\overline{\cat{C}}\) by Corollary~\ref{c:perp-equiv}).
By Corollary~\ref{c:subcat-substr}, we therefore find extriangulated substructures \(\EE_{\lperp{\cat{D}}}\) and \(\EE_{\rperp{\cat{D}}}\) on \(\overline{\cat{C}}\) (as well as the substructure \(\EE_{\cat{D}}\) on \(\cat{C}\) itself).
When \(\cat{C}\) is Calabi--Yau, as in our main application, we have \(\lperp{\cat{D}}=\rperp{\cat{D}}\), and so the two substructures on \(\overline{\cat{C}}\) coincide---we will use the notation \(\EE_{\rperp{\cat{D}}}\).

\section{Cluster-tilting and triangulations}
\label{s:ct-subcats}

Cluster-tilting subcategories are fundamental to the categorical and representation-theoretic approach to cluster algebras and related structures.
For triangulated categories, this goes back to foundational work of Buan, Marsh, Reineke, Reiten and Todorov \cite{BMRRT}.
We begin by recalling the relevant definitions.

\begin{defn}
Let \(\cat{T}\) be a \(\KK\)-linear additive category.
A \emph{\(\cat{T}\)-module} is a contravariant functor \(M\colon\cat{T}\to\Modcat{\KK}\).
We say that \(M\) is \emph{finitely generated} if there is an epimorphism \(\Hom_{\cat{T}}(\blank,T)\to M\) for some \(T\in\cat{T}\).

If \(\cat{T}\) is a full subcategory of a second \(\KK\)-linear additive category \(\cat{C}\), then we write \(\Hom_{\cat{C}}(\cat{T},X)=\Hom_{\cat{C}}(\blank,X)|_{\cat{T}}\) for any \(X\in\cat{C}\).
We say \(\cat{T}\) is \emph{contravariantly finite} in \(\cat{C}\) if \(\Hom_{\cat{C}}(\cat{T},X)\) is a finitely generated \(\cat{T}\)-module for all \(X\in\cat{T}\).
Dually, \(\cat{T}\) is \emph{covariantly finite} in \(\cat{C}\) if \(\Hom_{\cat{C}}(X,\cat{T})\) is a finitely generated \(\cat{T}^\op\)-module for all \(X\in\cat{C}\).
We say \(\cat{T}\) is \emph{functorially finite} if it is both contravariantly and covariantly finite.
\end{defn}

\begin{rem}
Note that \(\cat{T}\) is covariantly finite in \(\cat{C}\) if and only if \(\cat{T}^\op\) is contravariantly finite in \(\cat{C}^\op\).
These finiteness properties are often phrased in terms of the existence of approximations; by definition, there is an epimorphism \(\Hom_{\cat{C}}(\cat{T},T)\to\Hom_{\cat{C}}(\cat{T},X)\), for \(T\in\cat{T}\) and \(X\in\cat{C}\), if and only if there is a right \(\cat{T}\)-approximation \(T\to X\).
\end{rem}

\begin{defn}
For an extriangulated category \(\cat{C}\), we say \(X\in\cat{C}\) is \emph{rigid} if \(\EE_{\cat{C}}(X, X)=0\).
A full subcategory \(\cat{T}\subset\cat{C}\) is called
\begin{enumerate}
\item\emph{maximal rigid} if
\begin{align*}
\cat{T}&=\{X\in\cat{C}:\text{\(X\) is rigid and \(\EE_{\cat{C}}(X,T)=0=\EE_{\cat{C}}(T,X)\) for all \(T\in\cat{T}\)}\},
\end{align*}
\item \emph{weak cluster-tilting} if
\begin{align*}
\cat{T}&=\{X\in\cat{C}:\text{\(\EE_{\cat{C}}(X,T)=0\) for all \(T\in\cat{T}\)}\}\\
&=\{X\in\cat{C}:\text{\(\EE_{\cat{C}}(T,X)=0\) for all \(T\in\cat{T}\)}\},
\end{align*}
and
\item \emph{cluster-tilting} if it is weak cluster-tilting and functorially finite.
\end{enumerate}
We call \(\cat{C}\) \emph{weakly \(2\)-Calabi--Yau} if
\[\EE_{\cat{C}}(X,Y)=0\iff\EE_{\cat{C}}(Y,X)=0\]
for any \(X,Y\in\cat{C}\).
In particular, this holds if \(\cat{C}\) is a \(2\)-Calabi--Yau triangulated category.
If \(\cat{C}\) is weakly \(2\)-Calabi--Yau, the second equality in the definition of weak cluster-tilting becomes automatic.
\end{defn}

By definition, a cluster-tilting subcategory is weak cluster-tilting, and a weak cluster-tilting subcategory is maximal rigid.
The reverse implications do not hold in general, although Zhou and Zhu \cite[Thm.~2.6]{ZhoZhu} show that if \(\cat{C}\) is a \(2\)-Calabi--Yau triangulated category and has a cluster-tilting subcategory, then any functorially finite maximal rigid subcategory of \(\cat{C}\) is cluster-tilting.
This implies the same result for stably \(2\)-Calabi--Yau Frobenius extriangulated categories in general.
A very special situation in which maximal rigid subcategories are automatically weak cluster-tilting is the following.

\begin{prop}
\label{p:max-rigid-to-wct}
Let \(\cat{C}\) be a Krull--Schmidt and weakly \(2\)-Calabi--Yau extriangulated category.
If every indecomposable object of \(\cat{C}\) is rigid, then any maximal rigid subcategory of \(\cat{C}\) is weak cluster-tilting.
\end{prop}
\begin{proof}
Let \(\cat{T}\) be a maximal rigid subcategory of \(\cat{C}\), so there is an inclusion
\[
\cat{T}\subseteq\{X\in\cat{C}:\text{\(\EE_{\cat{C}}(T,X)=0\) for all \(T\in\cat{T}\)}\}=\{X\in\cat{C}:\text{\(\EE_{\cat{C}}(X,T)=0\) for all \(T\in\cat{T}\)}\},
\]
where for the equality we use that \(\cat{C}\) is weakly \(2\)-Calabi--Yau.

Choose \(X\in\cat{C}\) such that \(\EE_{\cat{C}}(T,X)=0\), and hence also \(\EE_{\cat{C}}(X,T)=0\), for all \(T\in\cat{T}\).
Since \(\cat{C}\) is Krull--Schmidt, we may choose a decomposition \(X=\bigoplus_{i=1}^nX_i\) in which each object \(X_i\) is indecomposable, hence rigid.
Since
\[0=\EE_{\cat{C}}\Bigl(T,\bigoplus_{i=1}^nX_i\Bigr)=\bigoplus_{i=1}^n\EE_{\cat{C}}(T,X_i),\]
we have for any \(1\leq i\leq n\) that \(\EE_{\cat{C}}(T,X_i)=0\).
Thus \(\EE_{\cat{C}}(X_i,T)=0\), and so \(X_i\in\cat{T}\) since \(\cat{T}\) is maximal rigid.
In particular, this means that \(\EE_{\cat{C}}(X_i,X_j)=0\) for all \(1\leq i, j\leq n\), so \(X\) is rigid and hence \(X\in\cat{T}\).
\end{proof}

In many examples, such as the original cluster categories of \cite{BMRRT}, all weak cluster-tilting subcategories are additively finite, meaning they are given by the additive closure \(\add{T}\) of some object \(T\in\cat{C}\).
Thus they are automatically cluster-tilting if \(\cat{C}\) is additionally Hom-finite.
In our context, we will be most interested in those examples in which the weak cluster-tilting subcategories do not have additive generators; sometimes such categories \(\cat{C}\), and their cluster-tilting subcategories \(\cat{T}\), are said to be of \emph{infinite rank}.
Interesting examples of such categories can be found in many other sources \cite{IguTod-CyclPoset, StoVRoo, HolJor-InftyGon}, and they are related to infinite rank cluster algebras \cite{Gratz-Colimits, GraGra-InfRank, GraGra-Colimits, CanFel}.

An important source of \(2\)-Calabi--Yau triangulated categories is the combinatorics of arcs in marked surfaces.
We focus on the case of discs, in which the assumptions of Proposition~\ref{p:max-rigid-to-wct} are satisfied.

\begin{defn}
\label{d:arc}
Given an oriented disc \(\disc\) with a set of marked points \(\pts\) in its boundary, an \emph{arc} of \((\disc,\pts)\) is a curve in \(\disc\) with endpoints in \(\pts\) such that the endpoints are distinct and not adjacent in \(\pts\) (i.e.\ any curve in \(\bdry\disc\) between these points must contain a third point of \(\pts\)).
Arcs are considered up to homotopy fixing \(\pts\). 
\end{defn}

Arcs in the sense of Definition~\ref{d:arc} are in natural bijection with unordered pairs of distinct and non-adjacent points of \(\pts\).
However, we wish to use concepts such as transverse intersection which have natural geometric definitions in terms of the representing curves.

Of particular interest for us are two families \(\ITcat_n\) and \(\PYcat_n\) of triangulated categories, which we will now describe, associated to discs with an infinite set \(\pts_n\) of marked points on their boundaries.
The set \(\pts_n\) has \(n\) accumulation points, all of which are two-sided (cf.~Figure~\ref{f:fountain}) and outside \(\pts_n\).
The category \(\ITcat_n\) will reflect the combinatorics of arcs between points of \(\pts_n\), whereas \(\PYcat_n\) will reflect that of arcs between points of the closure \(\close{\pts}_n\), i.e.\ the union of \(\pts_n\) with its accumulation points.

In describing these categories, we will repeatedly use the fact that \(\pts_n\) and \(\close{\pts}_n\) are cyclically ordered, in the following way.
Let \(\pts\) be any subset of \(\bdry\disc\), and consider the orientation-preserving universal cover \(\pi\colon\RR\to\bdry\disc\).
By convention, we take \(\RR\) to be oriented from negative to positive, and \(\bdry\disc\) to be oriented anticlockwise.
Then \(\univ{\pts}=\pi^{-1}(\pts)\) is totally ordered, via the usual linear ordering of \(\RR\), and thus \(\pts\) itself inherits a compatible cyclic ordering.

\begin{defn}
\label{d:cycl-ineq}
Given points \(x\) and \(y_i\) in \(\pts\), for \(i\in\NN\), we write
\[x\leq y_1\leq y_2\leq\dotsb\leq x^+\]
if there exist lifts \(\univ{x}\) and \(\univ{y}_i\) to \(\univ{\pts}\) such that
\[\univ{x}\leq\univ{y}_1\leq\univ{y}_2\leq\dotsb\leq\sigma\univ{x},\]
where \(\sigma\) generates the deck transformation group of \(\pi\colon\RR\to\bdry\disc\).
In other words, making a single anticlockwise circuit of \(\bdry\disc\) from \(x\) to \(x\), the points \(y_1,y_2,\dotsc\) appear in this order.
\end{defn}

With the notation from Definition~\ref{d:cycl-ineq}, the arc \(\gamma\) with endpoints \(x\) and \(y\) crosses the arc \(\gamma'\) with endpoints \(u\) and \(v\) if and only if
\[x<u<y<v<x^+,\]
as in Figure~\ref{f:Ptolemy}.
\begin{figure}
\begin{tikzpicture}[scale=2.4,every to/.style={hyperbolic disc}]
\draw (135:1) edge node[above, pos=0.4] {\(\gamma\)} (-45:1);
\draw (45:1) edge node[above, pos=0.4] {\(\gamma'\)} (-135:1);
\draw (135:1) to (45:1);
\draw (-45:1) to (45:1);
\draw (-45:1) to (-135:1);
\draw (-135:1) to (135:1);
\node at (90:0.55) {\(\alpha_1\)};
\node at (0:0.55) {\(\beta_2\)};
\node at (-90:0.55) {\(\alpha_2\)};
\node at (180:0.55) {\(\beta_1\)};
\draw[gray,very thick] (0,0) circle(1);
\draw (135:1.1) node {\(x\)};
\draw (45:1.1) node {\(v\)};
\draw (-45:1.1) node {\(y\)};
\draw (-135:1.1) node {\(u\)};
\end{tikzpicture}
\caption{A quadrilateral configuration, inducing two triangles of \(\ITcat_n\) as described in Theorem~\ref{t:ITcat}.
The curves \(\alpha_i\) and \(\beta_i\) are either arcs, or homotopic to an unmarked boundary segment.}
\label{f:Ptolemy}
\end{figure}

\begin{defn}
\label{d:succ}
A point \(x\in \pts\) has \emph{successor} \(z\ne x\) if whenever
\[x\leq y<z<x^+,\]
we have \(y=x\).
For each \(x\in \pts\) we write \(\susp x\) for the successor of \(x\) if it exists, and \(\susp x=x\) otherwise.
 If \(\gamma\) is an arc connecting \(x\) and \(y\) in \(\pts\), we define \(\susp\gamma\) to be the unique arc connecting \(\susp x\) and \(\susp y\).
\end{defn}

The map \(\susp\) from Definition~\ref{d:succ} is a permutation of \(\pts\).
 We observe that points \(x\ne y\) are adjacent, as in Definition~\ref{d:arc}, if and only if \(\susp x=y\) or \(\susp x=y\).
On \(\pts=\close{\pts}_n\), the fixed points of \(\susp\) are precisely the accumulation points \(\close{\pts}_n\setminus \pts_n\), and so \(\susp\) restricts to the analogously defined permutation of \(\pts_n\) (with no fixed points).

As a special case of their general construction of cluster categories for cyclic posets, Igusa and Todorov associate a triangulated category \(\ITcat_n\) to \((\disc,\pts_n)\) with the following properties.

\begin{thm}[{\cite[Thm.~2.5.3]{IguTod-CyclPoset}}]
\label{t:ITcat}
There is a \(\KK\)-linear Krull--Schmidt \(2\)-Calabi--Yau triangulated category \(\ITcat_n\) with indecomposable objects given by arcs in \((\disc,\pts_n)\), such that the suspension functor \(\Sigma\) acts on arcs by the permutation described above (also called \(\Sigma\)), and 
\[\Hom_{\ITcat_n}(\gamma,\susp\delta)=\begin{cases}\KK,&\text{\(\gamma\) and \(\delta\) cross transversely,}\\0,&\text{otherwise.}\end{cases}\]
Moreover, for each configuration as in Figure~\ref{f:Ptolemy}, there are distinguished triangles
\[\begin{tikzcd}
\gamma'\arrow{r}&\alpha_1\dsum\alpha_2\arrow{r}&\gamma\arrow{r}&\susp\gamma',
\end{tikzcd}
\quad
\begin{tikzcd}
\gamma\arrow{r}&\beta_1\dsum\beta_2\arrow{r}&\gamma'\arrow{r}&\susp\gamma.
\end{tikzcd}\]
Here we take \(\alpha_i\) or \(\beta_i\) to be zero if the corresponding curve is homotopic to an unmarked boundary segment.
\end{thm}

\begin{rem}
The complete set of distinguished triangles in \(\ITcat_{n}\) is generated by those described in Theorem~\ref{t:ITcat} (i.e.\ it is obtained by closing this set of triangles under isomorphism, rotation, pullback and pushout).
See \cite[Thm.~4.1]{CSP-TorPairs} for \(\ITcat_1\) (which is generated by any one of its arcs, each of which is a \(2\)-spherical object); the argument generalises to \(\ITcat_n\) for all \(n\) (cf.~\cite[\S6.1]{Franchini-TorPairs}, \cite[Lem.~4.16]{GraZvo}).

The composition map \(\circ\colon\Hom_{\ITcat_n}(\beta,\gamma)\times\Hom_{\ITcat_n}(\alpha,\beta)\to\Hom_{\ITcat_n}(\alpha,\gamma)\) is either zero or multiplication in \(\KK\), depending on the arcs \(\alpha\), \(\beta\) and \(\gamma\).
The situations in which the map is non-zero are described in \cite[Lem.~2.4.2]{IguTod-CyclPoset}, see also Lemma~\ref{l:hourglass} below.
\end{rem}

Each set \(\pts_n\) consists of \(n\) intervals, separated by the accumulation points (which lie outside \(\pts_n\)), and we choose an anticlockwise numbering of these intervals by the cyclically ordered set \(\ZZ_n=\{1,\dotsc,n\}\), as shown in Figure~\ref{f:M4} for \(n=4\).
In particular, this allows to discuss even and odd-numbered intervals.
\begin{figure}
\begin{tikzpicture}[scale=2.4,every to/.style={hyperbolic disc}]
\draw (125:1) to (200:1);
\draw (45:1) to (270:1);
\draw (-20:1) to (190:1);
\node at (155:0.65) {\(\delta\)};
\node at (55:0.6) {\(\gamma\)};
\node at (220:0.4) {\(\beta\)};
\draw[gray,very thick] (0,0) circle(1);
\foreach \n in {1,...,4}
{ \draw (115+90*\n:1) circle(0.0375) [fill=white];
  \draw (-20+90*\n:1.15) node {\n}; }
\end{tikzpicture}
\caption{A schematic of \((\disc,\pts_4)\).
Only the accumulation points are shown (as hollow circles, since they are not points in \(\pts_4\)), and the intervals between the accumulation points are each homeomorphic to \(\ZZ\subset\RR\).
The arc \(\delta\) is in \(\cat{D}\), whereas \(\gamma\) is in \(\rperp{\cat{D}}\).
However, \(\beta\) is not in \(\cat{D}\), since its endpoints are in different intervals, and it is not in \(\rperp{\cat{D}}\) either, since it crosses \(\delta\).}
\label{f:M4}
\end{figure}
The next definition is due to Paquette and Yıldırım \cite{PaqYil}.

\begin{defn}[{\cite[\S3]{PaqYil}}]
\label{d:PY-cat}
Let \(\cat{D}\subseteq\ITcat_{2n}\) be the subcategory additively generated by those arcs with both endpoints in the same even-numbered interval.
Since \(\susp\) restricts to a permutation of each interval, \(\cat{D}\) is a thick subcategory of \(\ITcat_{2n}\), and so we may define
\[\PYcat_n\defeq\ITcat_{2n}/\cat{D}\]
to be the Verdier quotient of \(\ITcat_{2n}\) by \(\cat{D}\).
\end{defn}

\begin{rem}
This follows related work by Fisher \cite{Fisher}, who gives a different description of \(\PYcat_1\) as the closure of \(\ITcat_1\) under homotopy colimits.
August--Cheung--Faber--Gratz--Schroll \cite{ACFGS1} describe a Frobenius exact model for \(\PYcat_1\) and relate it to the Grassmannian of \(2\)-dimensional quotients of a vector space of countably infinite dimension.
It follows from \cite[Prop.~5.12, Thm.~5.54]{Kalck-Thesis} that \(\PYcat_1\) is equivalent to the relative singularity category of an even-dimensional \(\type{A}_1\)-singularity.
Paquette--Yıldırım's categories, for all values of \(n\), have been further studied on a homological level by Franchini \cite{Franchini-TorPairs}, who classifies their torsion pairs and (co-)t-structures, and Murphy \cite{Murphy-GrothGrp,Murphy-OrlovSpec}, who calculates their Grothendieck groups and Orlov spectra.
\end{rem}

\begin{lem}
\label{l:dperp}
For \(\cat{D}\) as in Definition~\ref{d:PY-cat}, the indecomposable objects of \(\rperp{\cat{D}}\) are those with both endpoints in odd-numbered intervals of \(\pts_{2n}\).
\end{lem}
\begin{proof}
Because \(\cat{D}\) is a thick subcategory, we have \(\cat{D}=\susp^{-1}\cat{D}\).
By construction, indecomposable objects of \(\rperp{\cat{D}}=\rperp{(\susp^{-1}\cat{D})}\) are those arcs having no crossings with any of the arcs in \(\cat{D}\), which are precisely those described in the statement.
\end{proof}

\begin{eg}
A schematic of the Auslander--Reiten quiver of \(\ITcat_4\), and its subcategories \(\cat{D}\) and \(\rperp{\cat{D}}\), is shown in Figure~\ref{f:ITcat4}.
In this figure, morphisms are generated by straight lines with slope \(\pm1\), from left to right, and any parallelogram with sides of these slopes corresponds to a mesh-like relation.
Some care must be taken with this figure, since it only captures `large-scale' information.
For example, a morphism given by a straight line from one edge of the strip to the other represents a morphism \(\gamma\to\Sigma^2\gamma\) for some arc \(\gamma\), despite the fact that the two ends of this arc appear to be identified when passing to the Möbius band.

\begin{figure}
\begin{tikzpicture}[scale=2]
\pgfmathsetmacro{\triht}{0.5*tan(45)}
\foreach \n/\f in {0/black!40,1/white,2/black!40}
{\draw[fill=\f,dashed] (0+\n*1.2,0) -- (1+\n*1.2,0) -- (0.5+\n*1.2,\triht) -- cycle;}
\foreach \n/\f in {0/black!40,1/white,2/black!40}
{\draw[fill=\f,dashed] (0+\n*1.2,0.4+4*\triht) -- (1+\n*1.2,0.4+4*\triht) -- (0.5+\n*1.2,0.4+3*\triht) -- cycle;}
\foreach \n/\f in {0/white,1/black!15,2/white}
{\draw[fill=\f,dashed] (0+\n*1.2,0.2+2*\triht) -- (0.5+\n*1.2,0.2+3*\triht) -- (1+\n*1.2,0.2+2*\triht) -- (0.5+\n*1.2,0.2+\triht) -- cycle;}
\foreach \n/\f in {0/black!15,1/black!15}
{\draw[fill=\f,dashed] (0.6+\n*1.2,0.1+\triht) -- (1.1+\n*1.2,0.1+2*\triht) -- (1.6+\n*1.2,0.1+\triht) -- (1.1+\n*1.2,0.1) -- cycle;}
\foreach \n/\f in {0/black!15,1/black!15}
{\draw[fill=\f,dashed] (0.6+\n*1.2,0.3+3*\triht) -- (1.1+\n*1.2,0.3+4*\triht) -- (1.6+\n*1.2,0.3+3*\triht) -- (1.1+\n*1.2,0.3+2*\triht) -- cycle;}
\foreach \x/\y/\l in {0/0/44, 1/0/33, 2/0/22, 0/1/22, 1/1/11, 2/1/44}
{\draw (0.5+1.2*\x,0.17+0.06*\y+4*\triht*\y) node {\(\l\)};}
\foreach \x/\y/\l in {0/0/34, 1/0/23, 0/1/12, 1/1/14}
{\draw (1.1+1.2*\x,0.1+\triht+0.2*\y+2*\triht*\y) node {\(\l\)};}
\foreach \x/\l in {0/13, 1/24, 2/13}
{\draw (0.5+1.2*\x,0.2+2*\triht) node {\(\l\)};}
\draw (0.5,-0.1) edge [dotted,black!60] (0.5,0.5+4*\triht);
\draw (2.9,-0.1) edge [dotted,black!60] (2.9,0.5+4*\triht);
\end{tikzpicture}
\caption{A schematic of the Auslander--Reiten quiver of the category \(\ITcat_4\), drawn on a Möbius band.
Each triangle represents a \(\ZZ\type{A}_\infty\) component, and each diamond a \(\ZZ\type{A}_\infty^\infty\) component.
The components are characterised by the pair of intervals (as in Figure~\ref{f:M4}) in which their arcs have endpoints, and we label each component by this pair.
The darkly shaded components are those of \(\cat{D}\), whereas the unshaded components are those of \(\rperp{\cat{D}}\).}
\label{f:ITcat4}
\end{figure}

The corresponding figure for the Verdier quotient \(\PYcat_2\) is shown in Figure~\ref{f:PYcat2}.
Once again morphisms are generated by lines of slope \(\pm1\), satisfying mesh relations given by parallelograms.
The suspension acts on \(\type{A}_\infty^\infty\) components by translation against the direction of the arrows, and fixes the single-object component.
This means in particular that this indecomposable object is not rigid in the triangulated category \(\PYcat_n\) (cf.~Proposition~\ref{p:PY-exts} below).

\begin{figure}
\begin{tikzpicture}[scale=2]
\pgfmathsetmacro{\triht}{0.5*tan(45)}
\foreach \n in {0,1}
{\draw[dashed] (\n*1.2,0.1+\triht) -- (0.5+\n*1.2,0.1+2*\triht) -- (1+\n*1.2,0.1+\triht) -- (0.5+\n*1.2,0.1) -- cycle;}
\draw[dashed] (0.6,0) -- (1.6,0) -- (1.1,\triht) -- cycle;
\draw[dashed] (0.6,0.2+2*\triht) -- (1.6,0.2+2*\triht) -- (1.1,0.2+\triht) -- cycle;
\foreach \x/\y/\l in {0/0/13, 1/0/13}
{\draw (0.5+1.2*\x,0.1+\triht+0.2*\y+2*\triht*\y) node {\(\l\)};}
\draw (0.55,0.05) edge (1.05,0.05+\triht);
\draw (1.65,0.05) edge (1.15,0.05+\triht);
\draw (0.55,0.15+2*\triht) edge (1.05,0.15+\triht);
\draw (1.65,0.15+2*\triht) edge (1.15,0.15+\triht);
\draw (0.45,0.05) edge (-0.05,0.05+\triht);
\draw (-0.05,0.15+\triht) edge (0.45,0.15+2*\triht);
\draw (1.75,0.15+2*\triht) edge (2.25,0.15+\triht);
\draw (1.75,0.05) edge (2.25,0.05+\triht);
\foreach \x in {-0.1,1.1,2.3}
{\draw[fill=black] (\x,0.1+\triht) circle(0.015);}
\draw (1.1,0.17) node {\(33\)};
\draw (1.1,2*\triht+0.2-0.17) node {\(11\)};
\draw (0.5,-0.1) edge [dotted,black!60] (0.5,0.3+2*\triht);
\draw (1.7,-0.1) edge [dotted,black!60] (1.7,0.3+2*\triht);
\end{tikzpicture}
\caption[]{A schematic of the category \(\PYcat_2=\ITcat_4/\cat{D}\), drawn on a Möbius band.
Triangles and diamonds have the same meaning as in Figure~\ref{f:ITcat4}, while a line represents an \(\type{A}_\infty^\infty\) component, consisting of arcs with exactly one endpoint at an accumulation point of \(\close{\pts}_2\).
A dot indicates a component with a single indecomposable object.
Unlike \(\ITcat_n\), the category \(\PYcat_n\) does not have Auslander--Reiten triangles (see Remark~\ref{r:2CY}), but our schematic still represents a quiver with a vertex for each indecomposable object and arrows representing irreducible morphisms (but with no translation structure).}
\label{f:PYcat2}
\end{figure}
\end{eg}

Indecomposable objects of \(\PYcat_n\) are related to arcs between points of \(\close{\pts}_n\), as we now explain.
First, define a map \(\pts_{2n}\to\close{\pts}_n\), denoted by \(p\mapsto\squeeze{p}\), such that for each \(k\in\ZZ_n\)
\begin{enumerate}
\item interval \(2k-1\) in \(\pts_{2n}\) is mapped to interval \(k\) in \(\pts_n\subset\close{\pts}_n\) via an order preserving bijection, and
\item every point in interval \(2k\) is mapped to the accumulation point in \(\close{\pts}_n\) between intervals \(k\) and \(k+1\).
\end{enumerate}
We observe that \(\squeeze{\susp p}=\susp\squeeze{p}\) for all \(p\in \pts_{2n}\).
Thus, if \(p\) and \(q\) are non-adjacent in \(\pts_{2n}\), then \(\squeeze{p}\) and \(\squeeze{q}\) are either equal or non-adjacent in \(\close{\pts}_n\).
This map on marked points can be extended to a continuous map \(\disc\to\disc\), collapsing \(n\) cones (meeting the boundary in the odd-numbered intervals) to lines (meeting the boundary at an accumulation point of \(\close{\pts}_n\)).

If \(\gamma\) is the arc with endpoints \(u\) and \(v\) in \(\pts_{2n}\), and \(u\) and \(v\) do not lie in the same even-numbered interval, we write \(\squeeze{\gamma}\) for the arc with endpoints \(\squeeze{u}\) and \(\squeeze{v}\) in \(\close{\pts}_{n}\), our assumption on \(u\) and \(v\) ensuring that \(\squeeze{u}\) and \(\squeeze{v}\) are distinct, and the observation above ensuring that they are non-adjacent.
This observation also implies that \(\squeeze{\susp\gamma}=\susp\squeeze{\gamma}\).

\begin{prop}[{\cite[Cor.~3.11]{PaqYil}}]
\label{p:PY-bijection}
The map \(\gamma\mapsto\squeeze{\gamma}\), defined on arcs \(\gamma\) with endpoints in \(\pts_{2n}\) not lying in the same even-numbered interval, is a bijection between isomorphism classes of indecomposable objects in \(\PYcat_n\) and arcs in \((\disc,\close{\pts}_n)\).
\end{prop}

\begin{rem}
\label{r:dperp}
While \(\ITcat_{2n}\) is a skeletal category, meaning that two objects are isomorphic if and only if they are equal, this is not the case for \(\PYcat_{n}\).
Part of the content of Proposition~\ref{p:PY-bijection} is that arcs \(\gamma\) and \(\delta\) in \(\ITcat_{2n}\) are isomorphic in \(\PYcat_n\) if and only if \(\squeeze{\gamma}=\squeeze{\delta}\).
This equivalence relation on arcs between points of \(\pts_{2n}\) is called \emph{similarity} in \cite{PaqYil}, and there are bijections between similarity classes of arcs in \((\disc,\pts_{2n})\), arcs in \((\disc,\close{\pts}_n)\), and isoclasses of indecomposable objects in \(\PYcat_n\).

By Lemma~\ref{l:dperp}, if \(\gamma\) is an indecomposable object of \(\rperp{\cat{D}}\subseteq\ITcat_{2n}\) then \(\squeeze{\gamma}\) is an arc between points of \(\pts_n\), i.e.\ it avoids the accumulation points.
In this case, the isomorphism class of \(\gamma\) in \(\PYcat_n\) consists of a single object.
The arcs \(\squeeze{\gamma}\) obtained in this way are precisely the objects of \(\ITcat_n\), which thus include into \(\PYcat_n\) as the objects of the orthogonal category \(\rperp{\cat{D}}\).
This can be seen in Figure~\ref{f:PYcat2}, in which the four labelled components (of types \(\ZZ\type{A}_\infty\) and \(\ZZ\type{A}_\infty^\infty\)) comprise \(\rperp{\cat{D}}\).
\end{rem}

In contrast to \(\ITcat_{2n}\), the quotient category \(\PYcat_n\) is not even weakly \(2\)-Calabi--Yau.
Indeed, Paquette and Yıldırım compute the extension space between any pair of indecomposable objects as follows.

\begin{prop}[{\cite[Prop.~3.14]{PaqYil}}]
\label{p:PY-exts}
Let \(\gamma\) and \(\delta\) be arcs in \((\disc,\pts_{2n})\).
Then we have \(\Hom_{\PYcat_n}(\gamma,\susp \delta)=\KK\) if and only if either
\begin{enumerate}
\item\label{0-limit-end} \(\squeeze{\gamma}\) and \(\squeeze{\delta}\) cross transversely,
\item\label{1-limit-end} \(\squeeze{\gamma}\) and \(\squeeze{\delta}\) are distinct, meet at an accumulation point \(p\) of \(\close{\pts}_n\), and the angle from \(\squeeze{\gamma}\) to \(\squeeze{\delta}\) at \(p\) on the inside of the disc is clockwise (see Figure~\ref{f:1-limit-end}), or
\item\label{2-limit-end} \(\squeeze{\gamma}=\squeeze{\delta}\) has both endpoints at accumulation points of \(\close{\pts}_n\).
\end{enumerate}
In all other cases, \(\Hom_{\PYcat_n}(\gamma,\susp\delta)=0\).
\end{prop}
\begin{figure}
\begin{tikzpicture}[scale=10,every to/.style={hyperbolic disc}]
\draw [gray,very thick,domain=60:120] plot ({cos(\x)}, {sin(\x)});
\node[label={\(p\)}] (p) at (0,1) {};
\draw [black] (p) circle(0.01) [fill=black];
\draw (0.2,0.7) to (0,1);
\draw (0,1) to (-0.2,0.7);
\draw [-angle 90, dashed, domain=-75:-105] plot ({0.2*cos(\x)},{1+0.2*sin(\x)});
\node at (0.1,0.82) {\(\squeeze{\gamma}\)};
\node at (-0.1,0.82) {\(\squeeze{\delta}\)};
\end{tikzpicture}
\caption{Arcs in \((\disc,\close{\pts}_n)\) meeting at an accumulation point \(p\) in the configuration described in Proposition~\ref{p:PY-exts}\ref{1-limit-end}.
In this case, \(\Hom_{\PYcat_n}(\gamma,\susp\delta)=\KK\) but \(\Hom_{\PYcat_n}(\delta,\susp\gamma)=0\).}
\label{f:1-limit-end}
\end{figure}

Note that condition \ref{1-limit-end} of Proposition~\ref{p:PY-exts} is not symmetric in \(\gamma\) and \(\delta\), and indeed can hold for at most one of the two orderings.
This witnesses the failure of \(\PYcat_n\) to be weakly \(2\)-Calabi--Yau.

\begin{rem}
Combining Remark~\ref{r:dperp} and Proposition~\ref{p:PY-exts}, we see that the subcategory \(\rperp{\cat{D}}\subseteq\ITcat_{2n}\) and its image in \(\PYcat_n\) are equivalent to the category \(\ITcat_n\).
\end{rem}

\begin{rem}
\label{r:KS}
The previous two results of Paquette--Yıldırım together imply that \(\PYcat_n\) is Krull--Schmidt.
Indeed, they imply that \(\dim\Hom_{\PYcat_n}(\gamma,\gamma)\leq 1\), and hence this algebra is local, whenever \(\gamma\in\PYcat_n\) is indecomposable.
Since objects with local endomorphism algebra are always indecomposable, and \(\PYcat_n\) is Hom-finite, this is equivalent to the Krull--Schmidt property (see, e.g., \cite[Thm.~6.1]{Shah-KRS}).
In particular, \(\PYcat_n\) is idempotent complete, a property not enjoyed by Verdier quotients of idempotent complete categories in general.
\end{rem}

Paquette and Yıldırım use Proposition~\ref{p:PY-exts} to classify weak cluster-tilting subcategories in the triangulated category \(\PYcat_n\) in terms of triangulations of $(\disc,
\close{\pts}_n)$.
The conditions in Proposition~\ref{p:PY-exts} translate to strong restrictions on such a triangulation, since they imply that a non-crossing collection of arcs can give rise (via direct sum) to a non-rigid object in \(\PYcat_n\), in contrast to the case of \(\ITcat_n\).
Indeed, if \(\squeeze{\gamma}\) is an arc between accumulation points, then the indecomposable object \(\gamma\in\PYcat_n\) already fails to be rigid.

In the remainder of the section, we will demonstrate that the dictionary between rigidity of objects and crossings of arcs may be restored by instead considering \(\PYcat_n\) as an extriangulated category with the extriangulated substructure \(\EE_{\rperp{\cat{D}}}\) from Section~\ref{s:substruct}, where \(\cat{D}\subseteq\ITcat_{2n}\) is the subcategory used to define \(\PYcat_n=\ITcat_{2n}/\cat{D}\).

\begin{eg}
\label{eg:non-conf}
Since \(\rperp{\cat{D}}\) is a proper subcategory of \(\PYcat_n\)---it does not contain any arc incident with an accumulation point---the extriangulated substructure \(\EE_{\rperp{\cat{D}}}\) has strictly fewer conflations than the original triangulated structure on \(\PYcat_n\).
Indeed, whenever \(\gamma\notin\rperp{\cat{D}}\), the triangle
\[\begin{tikzcd}\Sigma^{-1}\gamma\arrow{r}&0\arrow{r}&\gamma\arrow{r}{1}&\gamma\end{tikzcd}\]
is not a conflation for \(\EE_{\rperp{\cat{D}}}\), since the identity map on \(\gamma\) does not factor over \(\rperp{\cat{D}}\).
A less trivial example of a triangle from \(\PYcat_n\) which is not a conflation for \(\EE_{\rperp{\cat{D}}}\) is shown in Figure~\ref{f:non-conf}.
\begin{figure}
\begin{tikzpicture}[scale=2,every to/.style={hyperbolic disc}]
\draw (56:1) to (204:1);
\draw (204:1) to (270:1);
\draw (270:1) to (56:1);
\draw (-22:0.26) node {\(\alpha\)};
\draw (-125:0.69) node {\(\gamma\)};
\draw (135:0.24) node {\(\beta\)};
\draw (-30:0.11) edge[-angle 90] (145:0.1);
\draw (185:0.19) edge[-angle 90] (-130:0.52);
\draw[gray,very thick] (0,0) circle(1);
\foreach \n in {1,2}
{ \draw (90-180*\n:1) circle(0.025) [fill=black];}
\end{tikzpicture}
\quad\quad
\begin{tikzpicture}[scale=2.8]
\pgfmathsetmacro{\triht}{0.5*tan(45)}
\foreach \n in {0,1}
{\draw[dashed,black!50] (\n*1.2,0.1+\triht) -- (0.5+\n*1.2,0.1+2*\triht) -- (1+\n*1.2,0.1+\triht) -- (0.5+\n*1.2,0.1) -- cycle;}
\draw[dashed,black!50] (0.6,0) -- (1.6,0) -- (1.1,\triht) -- cycle;
\draw[dashed,black!50] (0.6,0.2+2*\triht) -- (1.6,0.2+2*\triht) -- (1.1,0.2+\triht) -- cycle;
\draw (0.55,0.05) edge[black!50] (1.05,0.05+\triht);
\draw (1.65,0.05) edge[black!50] (1.15,0.05+\triht);
\draw (0.55,0.15+2*\triht) edge[black!50] (1.05,0.15+\triht);
\draw (1.65,0.15+2*\triht) edge[black!50] (1.15,0.15+\triht);
\draw (0.45,0.05) edge[black!50] (-0.05,0.05+\triht);
\draw (-0.05,0.15+\triht) edge[black!50] (0.45,0.15+2*\triht);
\draw (1.75,0.15+2*\triht) edge[black!50] (2.25,0.15+\triht);
\draw (1.75,0.05) edge[black!50] (2.25,0.05+\triht);
\foreach \x in {-0.1,1.1,2.3}
{\draw[black!50,fill=black!50] (\x,0.1+\triht) circle(0.015);}
\draw (0.5,-0.1) edge [dotted,black!60] (0.5,0.3+2*\triht);
\draw (1.7,-0.1) edge [dotted,black!60] (1.7,0.3+2*\triht);
\node at (0.6,0.2+\triht) (beta) {\(\beta\)};
\node at (0.9,0.4) (gamma) {\(\gamma\)};
\node at (0.9+\triht,0.4+\triht) (Sigma-alpha) {\(\susp\alpha\)};
\node at (0.2,-0.2+\triht) (alpha) {\(\alpha\)};
\draw (alpha) edge[-angle 90,thick] (beta);
\draw (beta) edge[-angle 90, thick] (gamma);
\draw (gamma) edge[-angle 90, thick] (Sigma-alpha);
\end{tikzpicture}
\caption{The triangle \(\alpha\to\beta\to\gamma\to\susp\alpha\) in \(\PYcat_2\) is not a conflation for \(\EE_{\rperp{\cat{D}}}\), since its connecting morphism does not factor over \(\rperp{\cat{D}}\).
The rotations \(\susp^{-1}\gamma\to\alpha\to\beta\to\gamma\) and \(\susp^{-1}\beta\to\susp^{-1}\gamma\to\alpha\to\beta\) of this triangle are conflations for \(\EE_{\rperp{\cat{D}}}\), however, since their connecting morphisms factor (trivially) over \(\beta\in\rperp{\cat{D}}\).}
\label{f:non-conf}
\end{figure}
\end{eg}

To understand the conflations for \(\EE_{\rperp{\cat{D}}}\), we need to determine the morphisms in \(\PYcat_n\) factoring over the subcategory \(\rperp{\cat{D}}\), for which the next few results will be useful.

\begin{lem}
\label{l:1dim}
Let \(\cat{C}\) be a \(\KK\)-linear Krull--Schmidt category such that \(\dim\Hom_{\cat{C}}(X,Y)\leq 1\) for any indecomposable objects \(X,Y\in\cat{C}\), and let \(S\) be a set of indecomposable objects in \(\cat{C}\).
Then for \(X,Y\in\cat{C}\) indecomposable, a morphism \(f\colon X\to Y\) factors over \(\add{S}\) if and only if it factors over a single object from \(S\).
\end{lem}
\begin{proof}
All of the content of the statement is in the `only if' direction, in the case that \(f\ne0\).
Assume \(f\) factors as the composition
\[\begin{tikzcd}
X\arrow{r}{g}&\displaystyle\bigoplus_{i=1}^kS_i\arrow{r}{h}&Y
\end{tikzcd}\]
for some \(S_i\in S\).
Writing \(g_i\colon X\to S_i\) and \(h_i\colon S_i\to Y\) for the components of \(g\) and \(h\), this means that \(0\ne f=\sum_{i=1}^k h_ig_i\), and so \(h_jg_j\ne0\) for some \(j\).
Since \(\dim\Hom_{\cat{C}}(X,Y)=1\), the morphisms \(f\) and \(h_jg_j\) are linearly dependent, and so \(f\) factors over \(S_j\).
\end{proof}

\begin{defn}
Let \(\gamma\) and \(\delta\) be arcs in \((\disc,\pts_n)\) which either cross or meet at an endpoint.
Rotating \(\gamma\) clockwise onto \(\delta\) sweeps out two intervals in the boundary of the disc, as indicated in Figure~\ref{f:hourglass}; note that if \(\gamma\) and \(\delta\) meet at a marked point on the boundary of the disc then one of these intervals consists solely of this single marked point.
We write \([\gamma,\delta]\) for the set of arcs with one endpoint in each of these boundary intervals.
\begin{figure}
\begin{minipage}{0.75\textwidth}
\begin{tikzpicture}[baseline=0,scale=2.5,every to/.style={hyperbolic disc}]
\draw [fill=black!10] (60:1) to (0,0) to (135:1) arc (135:60:1) (60:1);
\draw [fill=black!10] (315:1) to (0,0) to (240:1) arc (240:315:1) (315:1);
\draw [fill=black] (0,0) circle (0.02);
\draw (60:1) edge[pos=.7, "\(\delta\)",very thick] (0,0); \draw (0,0) edge[very thick] (240:1);
\draw (135:1) edge[pos=.7, "\(\gamma\)",swap,very thick] (0,0); \draw (0,0) edge[very thick] (315:1);
\draw (115:1) edge["\(\alpha\)"] (0,0); \draw (0,0) edge (295:1);
\draw (-75:1) to (-112:1);
\draw (260:0.65) node {\(\beta\)};
\draw [-angle 90, dashed, domain=3:72] plot ({0.24*cos(135-\x)},{0.24*sin(135-\x)});
\draw [-angle 90, dashed, domain=3:72] plot ({0.24*cos(315-\x)},{0.24*sin(315-\x)});
\draw [gray,very thick] (0,0) circle (1);
\end{tikzpicture}\hfill
\begin{tikzpicture}[baseline=0,scale=2.5,every to/.style={hyperbolic disc}]
\draw [fill=black!10] (315:1) to (100:1) to (210:1) arc (210:315:1) (315:1);
\draw[very thick] (100:1) to (210:1);
\draw[very thick] (100:1) to (315:1);
\draw (100:1) to (265:1);
\draw (210:1) to (270:1);
\draw (225:0.45) node {\(\beta\)};
\draw (275:0.15) node {\(\alpha\)};
\draw (60:0.3) node {\(\gamma\)};
\draw (145:0.4) node {\(\delta\)};
\draw [-angle 90, dashed, domain=0:18] plot ({cos(100)+0.8*cos(-75-\x)},{sin(100)+0.8*sin(-75-\x)});
\draw [gray,very thick] (0,0) circle (1);
\end{tikzpicture}
\end{minipage}
\caption{Regions, and hence boundary intervals, swept out by rotating \(\gamma\) clockwise onto \(\delta\).
In each figure, \(\alpha\) is an example of an arc in \([\gamma,\delta]\), whereas \(\beta\) is not.}
\label{f:hourglass}
\end{figure}
\end{defn}

\begin{lem}
\label{l:hourglass}
Let \(\gamma\ne\delta\) be arcs in \((\disc,\pts_n)\).
If \(\Hom_{\ITcat_n}(\gamma,\delta)\ne0\) then \(\gamma\cap\delta\ne\varnothing\).
Moreover, a non-zero morphism \(f\colon\gamma\to\delta\) in \(\ITcat_n\) factors over an object in \(\add{S}\), for some set \(S\) of arcs, if and only if \(S\cap [\gamma,\delta]\ne\varnothing\).
\end{lem}
\begin{proof}
By Theorem~\ref{t:ITcat}, if \(\Hom_{\ITcat_n}(\gamma,\delta)\ne0\), then the arcs \(\gamma\) and \(\susp^{-1}\delta\) have a transverse intersection, which can only happen if \(\gamma\cap\delta\ne\varnothing\); while \(\gamma\) and \(\delta\) may not themselves have a transverse intersection, they must at least share an endpoint.

Now by Lemma~\ref{l:1dim}, a non-zero morphism \(f\colon\gamma\to\delta\) factors over \(\add{S}\) if and only if it factors over a single arc \(\alpha\in S\).
There must be intersections (possibly at endpoints) between \(\alpha\) and \(\gamma\) and between \(\alpha\) and \(\delta\) so that there exist non-zero maps \(\gamma\to\alpha\) and \(\alpha\to\delta\).
In this case, the two maps compose to a non-zero map \(\gamma\to\delta\) (and hence, after rescaling if necessary, to \(f\)) if and only if \(\alpha\in [\gamma,\delta]\), by \cite[Lem.~2.4.2]{IguTod-CyclPoset}, taking \(Z=M_n\) and \(\phi=\susp\).
Note for comparison that Igusa and Todorov's condition (2.1) becomes
\[x_0\leq y_0< \susp^{-1}x_1<x_0^+,\quad x_1\leq y_1<\susp^{-1}x_0<x_0^+,\]
with notation as in Definition~\ref{d:cycl-ineq}.
\end{proof}

Our main result is that we can recover a bijection between non-crossing collections of arcs and rigid objects by weakening the triangulated structure of \(\PYcat_n\) to the extriangulated structure \(\EE_{\rperp{\cat{D}}}\) from Corollary~\ref{c:subcat-substr}, exploiting the fact that \(\PYcat_n\) is defined as the Verdier quotient \(\ITcat_{2n}/\cat{D}\).

\begin{thm}
\label{t:Ext-cross}
Let \(\gamma\) and \(\delta\) be arcs in \((\disc,\pts_{2n})\).
Then
\begin{equation}
\label{eq:EE-formula}
\EE_{\rperp{\cat{D}}}(\gamma,\delta)=\begin{cases}\KK,&\text{\(\squeeze{\gamma}\) and \(\squeeze{\delta}\) cross transversely,}\\0,&\text{otherwise}.\end{cases}
\end{equation}
In particular, every arc is rigid in \((\PYcat_n,\EE_{\rperp{\cat{D}}})\).
\end{thm}
\begin{proof}
Recall that \(\EE_{\rperp{\cat{D}}}(\gamma,\delta)\leq\Hom_{\PYcat_n}(\gamma,\susp\delta)\).
If \(\Hom_{\PYcat_n}(\gamma,\susp\delta)=0\) then, by Proposition~\ref{p:PY-exts}, the arcs \(\squeeze{\gamma}\) and \(\squeeze{\delta}\) do not cross transversely.
In this case we must have \(\EE_{\rperp{\cat{D}}}(\gamma,\delta)=0\), and so \eqref{eq:EE-formula} is correct.

Thus we may assume that \(\Hom_{\PYcat_n}(\gamma,\susp\delta)=\KK\), so we are in one of the cases \ref{0-limit-end}, \ref{1-limit-end} or \ref{2-limit-end} from Proposition~\ref{p:PY-exts}.
Pick a non-zero map \(\varphi\colon\gamma\to\susp\delta\) in \(\PYcat_n\).
To finish the proof, we must show that in case \ref{0-limit-end} we have \(\varphi\in(\rperp{\cat{D}})\), so that \(\EE_{\rperp{\cat{D}}}(\gamma,\delta)=\KK\), whereas in cases \ref{1-limit-end} and \ref{2-limit-end} we have \(\varphi\notin(\rperp{\cat{D}})\), so \(\EE_{\rperp{\cat{D}}}(\gamma,\delta)=0\).
We do this case by case.
\medskip

\textit{Case \ref{0-limit-end}:} In this case \(\squeeze{\gamma}\) and \(\squeeze{\delta}\) cross transversely, which means that \(\gamma\) and \(\delta\) must do also, and hence \(\Hom_{\ITcat_{2n}}(\gamma,\susp\delta)\ne0\).
We first show that there is a map \(\hat{\varphi}\in\Hom_{\ITcat_{2n}}(\gamma,\susp\delta)\) which is sent to \(\varphi\) under the localisation functor \(\loc\colon\ITcat_{2n}\to\PYcat_n\).
This amounts to showing that the induced map \(\Hom_{\ITcat_{2n}}(\gamma,\susp\delta)\to\Hom_{\PYcat_n}(\gamma,\susp\delta)=\KK\) is non-zero, and hence surjective, which will follow from showing that a non-zero map \(\psi\colon\gamma\to\susp\delta\) in \(\ITcat_{2n}\) does not factor over \(\cat{D}\).

Let \(p\) be an endpoint of \(\gamma\) and \(q\) an endpoint of \(\delta\).
If \(p\) and \(q\) were in the same even-numbered interval then the arcs \(\squeeze{\gamma}\) and \(\squeeze{\delta}\) would meet at the corresponding accumulation point of \(\pts_n\), instead of crossing transversely.
Thus, no even interval contains an endpoint of both \(\gamma\) and \(\delta\).
Since \(\susp\delta\) has endpoints in the same intervals as \(\delta\), the same conclusion applies to \(\gamma\) and \(\susp\delta\).
In particular, this means that no arc in \([\gamma,\susp\delta]\) has endpoints in the same even-numbered interval, i.e.\ that \([\gamma,\susp\delta]\cap\cat{D}=\varnothing\).
It then follows from Lemma~\ref{l:hourglass} that \(\psi\) does not factor over \(\cat{D}\).

Thus, we may pick a lift \(\hat{\varphi}\) of \(\varphi\) from \(\PYcat_n\) to \(\ITcat_{2n}\).
To show that \(\varphi\) factors over \(\rperp{\cat{D}}\), and hence that \(\varphi\in\EE_{\rperp{\cat{D}}}(\gamma,\delta)=\KK\), it suffices to show that this is true of \(\hat{\varphi}\).
By the argument in the preceding paragraph, each of the two boundary segments swept out by rotating \(\gamma\) clockwise onto \(\delta\) must contain a point from an odd-numbered interval.
Choosing one such point from each boundary segment, the arc \(\alpha\) between them lies in both \(\rperp{\cat{D}}\) and \([\gamma,\susp\delta]\), the second property meaning that \(\hat{\varphi}\) factors over \(\alpha\) by Lemma~\ref{l:hourglass}.
\medskip

\textit{Case \ref{1-limit-end}:}  In this case, the arcs \(\squeeze{\gamma}\) and \(\squeeze{\delta}\) are arranged as in Figure~\ref{f:1-limit-end}.
If \(\varphi\in(\rperp{\cat{D}})\) then by Lemma~\ref{l:1dim} there is an arc \(\alpha\in\rperp{\cat{D}}\) such that \(\varphi\) factors over \(\alpha\) in \(\PYcat_n\), i.e.\ \(\varphi=\varphi_2\circ\varphi_1\) for morphisms \(\varphi_1\colon\gamma\to\alpha\) and \(\varphi_2\colon\alpha\to\susp\delta\) in \(\PYcat_n\).

Since \(\alpha\in\rperp{\cat{D}}=\lperp{\cat{D}}\), the morphisms \(\varphi_1\) and \(\varphi_2\) lift uniquely to morphisms \(\hat{\varphi}_1\) and \(\hat{\varphi}_2\) in \(\ITcat_{2n}\) by Proposition~\ref{p:Dperp}.
This yields a lift \(\hat{\varphi}\defeq\hat{\varphi}_2\circ\hat{\varphi}_1\) of \(\varphi\) which still factors over \(\alpha\).

Since \(\varphi\ne0\), we must also have \(\hat{\varphi}\ne0\), so \(\gamma\) and \(\susp\delta\) meet by Lemma~\ref{l:hourglass}. By the same lemma, the arc \(\alpha\) lies in \([\gamma,\susp\delta]\).
But one endpoint of \(\gamma\) must lie in the same even-numbered interval as one endpoint of \(\susp\delta\) (equivalently of \(\delta\)) because \(\squeeze{\gamma}\) and \(\squeeze{\delta}\) meet at an accumulation point. Moreover, \(\squeeze{\delta}\) is a clockwise rotation of \(\squeeze{\gamma}\) around this accumulation point, so the two endpoints of \(\gamma\) and \(\Sigma\delta\) in this common interval must bound one of the two clockwise boundary segments involved in the definition of \([\gamma,\delta]\). (Comparing to Figure~\ref{f:hourglass}, this means that the relevant pair of endpoints is either the north or the south pair, not the east or west pair.) Thus every arc in \([\gamma,\susp\delta]\) has one endpoint in this even interval, and hence \([\gamma,\susp\delta]\) contains no arcs from \(\rperp{\cat{D}}\) by Lemma~\ref{l:dperp}.
We conclude from this contradiction that \(\varphi\not\in(\rperp{\cat{D}})\).
\medskip

\textit{Case \ref{2-limit-end}:}  In this case \(\squeeze{\gamma}\) is an arc between accumulation points of \(\close{\pts}_n\), and so \(\susp\gamma=\gamma\) in \(\PYcat_n\).
It follows that \(\Hom_{\PYcat_n}(\gamma,\susp\gamma)=\Hom_{\PYcat_n}(\gamma,\gamma)=\KK\) is spanned by \(\id_{\gamma}\).
Since \(\gamma\notin\rperp{\cat{D}}\) (see Remark~\ref{r:dperp}), we have \(\id_{\gamma}\notin(\rperp{\cat{D}})\) as required.
\end{proof}

\begin{cor}
\label{c:w2cy}
For objects \(X,Y\in\PYcat_n\), we have \(\dim_{\KK}\EE_{\rperp{\cat{D}}}(X,Y)=\dim_{\KK}\EE_{\rperp{\cat{D}}}(Y,X)\).
In particular, the extriangulated category \((\PYcat_n,\EE_{\rperp{\cat{D}}})\) is weakly \(2\)-Calabi--Yau.
\end{cor}
\begin{proof}
This follows from Theorem~\ref{t:Ext-cross} together with the fact that \(\PYcat_n\) is Krull--Schmidt and intersection is a symmetric relation on arcs.
\end{proof}

\begin{rem}
\label{r:2CY}
We suspect that \((\PYcat_n,\EE_{\rperp{\cat{D}}})\) is \(2\)-Calabi--Yau in a stronger sense, namely that there is a functorial duality
\[\EE_{\rperp{\cat{D}}}(X,Y)=\Kdual\EE_{\rperp{\cat{D}}}(Y,X)\]
for all \(X,Y\in\PYcat_n\).
Note however that \((\PYcat_n,\EE_{\rperp{\cat{D}}})\) is not Frobenius as an extriangulated category; while the only projective or injective object is \(0\), not every object admits a deflation from \(0\), so there are not enough projectives (or, similarly, enough injectives).
In particular, \((\PYcat_n,\EE_{\rperp{\cat{D}}})\) is neither triangulated, nor admits a triangulated stable category on which the Calabi--Yau property could be expressed in terms of a Serre functor.
The failure of \((\PYcat_n,\EE_{\rperp{\cat{D}}})\) to be a Frobenius extriangulated category affects much of the rest of the paper and its sequel \cite{CKaP2}: it means we need to make bespoke arguments to establish several properties that are known to hold very generally for (weakly) \(2\)-Calabi--Yau Frobenius extriangulated categories.

Further to this, \((\PYcat_n,\EE_{\rperp{\cat{D}}})\) does not have Auslander--Reiten conflations \cite{INP} in general.
Indeed, the arc \(\gamma\) between the two accumulation points in \((\disc,\close{\pts}_2)\) is not the codomain of a non-split deflation in \((\PYcat_2,\EE_{\rperp{\cat{D}}})\), because the cone of such a morphism in the triangulated category \(\PYcat_n\) would have to be a non-zero map \(\gamma\to\gamma=\susp\Gamma\), all of which are isomorphisms and hence do not factor over \(\rperp{\cat{D}}\) (cf.~Example~\ref{eg:non-conf}).
In particular, there is no Auslander--Reiten conflation ending in \(\gamma\).
The subcategory \(\rperp{\cat{D}}\) does have Auslander--Reiten conflations, however, being equivalent to the \(2\)-Calabi--Yau triangulated category \(\ITcat_n\) as in Remark~\ref{r:dperp}.

Interestingly, a construction by Christ \cite{Christ-Fukaya} also produces non-Calabi--Yau triangulated categories associated to marked surfaces, in which the expected correspondence between extensions in the category and crossings of arcs in the surface is recovered by passing to an extriangulated substructure \cite[Lem.~7.16]{Christ-Fukaya}.
In this context, the resulting extriangulated category is shown to be \(2\)-Calabi--Yau in the true, functorial, sense \cite[Prop.~5.20]{Christ-Fukaya} (and this is, roughly speaking, the defining property of the extriangulated substructure which is used).
However, Christ's construction is rather different to ours, going via the theory of (exact) \(\infty\)-categories, and applies to surfaces with finite sets of boundary marked points, and so at this time we do not know a precise relationship between his results and ours.
\end{rem}

\begin{cor}
\label{c:wct}
The bijection from Proposition~\ref{p:PY-bijection} induces a bijection \(\cat{T}\mapsto\indec{\cat{T}}\) between weak cluster-tilting subcategories of the extriangulated category \((\PYcat_n,\EE_{\rperp{\cat{D}}})\) and triangulations of \((\disc,\close{\pts}_n)\).
\end{cor}
\begin{proof}
By Theorem~\ref{t:Ext-cross}, the bijection from Proposition~\ref{p:PY-bijection} between objects and arcs induces a bijection between maximal rigid subcategories of \((\PYcat_n,\EE_{\rperp{\cat{D}}})\) and maximal collections of non-crossing arcs, i.e.\ triangulations.
Since \(\PYcat_n\) is Krull--Schmidt by Remark~\ref{r:KS}, is weakly \(2\)-Calabi--Yau by Corollary~\ref{c:w2cy}, and every indecomposable object in \((\PYcat_n,\EE_{\rperp{\cat{D}}})\) is rigid by Theorem~\ref{t:Ext-cross} again, all of its maximal rigid subcategories are weak cluster-tilting by Proposition~\ref{p:max-rigid-to-wct}.
\end{proof}

\begin{rem}
In the special case of \(\PYcat_1\), August, Cheung, Faber, Gratz and Schroll \cite[Thm.~5.5]{ACFGS2} give a different classification of weak cluster-tilting subcategories, using an explicit Frobenius exact model for this triangulated category.
In this exact category, the weak cluster-tilting subcategories coincide with the maximal almost rigid subcategories.
Here a subcategory is said to be \emph{almost rigid} if any non-split extension between two of its indecomposable objects has indecomposable middle term (matching \cite[Defn.~6.1]{BGMS} for quiver representations).
Analogous Frobenius exact models of \(\PYcat_n\) are not currently known when \(n>1\).

\begin{rem}
In \cite{GHL}, Geiss, Hernandez and Leclerc describe an infinite rank cluster algebra isomorphic (up to technical modifications such as topological completion) to the Grothendieck ring of a certain category of representations of the shifted quantum affine algebra \(U_q^\mu(\widehat{\Lie{g}})\) associated to a simple Lie algebra \(\Lie{g}\).
An overview of this construction is given in \cite[\S9]{GHL}.

In the type \(\type{A}_1\) case, where \(\Lie{g}=\sl_2\), a more detailed description of the cluster algebra is available which makes use of the combinatorics of infinite marked surfaces.
This is phrased in \cite{GHL} in terms of triangulations of an `infinity-gon', but whereas this often refers to the marked surface \((\disc,\close{\pts}_1)\), by identifying \(\close{\pts}_1\) with \(\ZZ\cup\{\infty\}\), in \cite{GHL} one instead wants the set of marked points to be \(\ZZ\cup\{\pm\infty\}\).
A possible categorical model is given by a full subcategory of \(\PYcat_2\), the category associated to the completed disc with two accumulation points:
its indecomposable objects are the arcs with endpoints either in the interval labelled \(1\), or incident to at least one of the two accumulation points.
In Figure~\ref{f:PYcat2}, this full subcategory consists of the component labelled \(11\), together with the two adjacent \(\type{A}_\infty^\infty\) components, and the component with a single indecomposable object.

While this is a triangulated subcategory of \(\PYcat_2\), with its usual triangulated structure, the interest in \cite{GHL} is in triangulations featuring many arcs incident with the same accumulation point (e.g.\ \cite[Fig.~17]{GHL}), combinatorics which is better captured by our modified extriangulated structure \(\EE_{\rperp{\cat{D}}}\); see \cite[Rem.~9.31]{GHL}.
An important notion in \cite{GHL} is that of infinite (or transfinite) sequences of mutations, as also considered in \cite{BauGra} and \cite{CanFel}.
At present we do not know how to use the categorification to gain any new insight into such mutation sequences.
On the other hand, one advantage of categorification for studying cluster combinatorics is that there are many ways to relate two cluster-tilting subcategories without declaring an explicit sequence of mutations between them (or even requiring such a sequence to exist).
\end{rem}
\end{rem}

\section{Functorial finiteness}
\label{s:approx}

Our next goal is to characterise the triangulations which correspond to cluster-tilting subcategories under the bijection of Corollary~\ref{c:wct}, that is, those for which the corresponding subcategory is functorially finite.
For the categories \(\ITcat_n\), the analogous problem was solved by Gratz--Holm--Jørgensen \cite[Thm.~5.7]{GHJ}, primarily via conditions on the local structure of a triangulation around each accumulation point.
For \(\PYcat_n\), the situation turns out to be rather different: we need only rule out one particular sub-configuration of arcs.

\begin{defn}
\label{d:leapfrog}
In a marked disc \((\disc,\pts)\), an \emph{infinite leapfrog} (or \emph{infinite zig-zag} \cite{CanFel}) consists of two sets \(\{\alpha_i\}_{i\in I}\) and \(\{\beta_i\}_{i\in I}\) of arcs, indexed by the set \(I\in\{\ZZ,\ZZ_{\leq0},\ZZ_{\geq0}\}\), such that
\begin{enumerate}
\item all arcs from \(\{\alpha_i,\beta_i\}_{i\in I}\) are distinct, and none cross;
\item\label{d:leapfrog-alpha} each arc \(\alpha_i\) is incident with one endpoint of \(\beta_i\) and, if \(i-1\in I\), one endpoint of \(\beta_{i-1}\);
\item\label{d:leapfrog-beta} each arc \(\beta_i\) is incident with one endpoint of \(\alpha_i\) and, if \(i+1\in I\), one endpoint of \(\alpha_{i+1}\);
\item there is a curve \(\gamma\) (not necessarily between marked points) crossing all the arcs \(\alpha_i\) and \(\beta_j\).
\end{enumerate}
\begin{figure}
\begin{tikzpicture}[scale=3,every to/.style={hyperbolic disc}]
\draw (90:1) edge[pos=0.15, "\(\gamma\)", dashed] (270:1);
\foreach \t in {1,2,3,4}
{\draw (125+20*\t:1) to (45-20*\t:1);
\draw (45-20*\t:1) to (145+20*\t:1);}
\foreach \t in {-2,0,2}
{\draw (135+\t:0.933333) node {\(\cdot\)};
\draw (30+\t:0.933333) node {\(\cdot\)};
\draw (233+\t:0.933333) node {\(\cdot\)};
\draw (-43+\t:0.933333) node {\(\cdot\)};}
\draw (22:0.45) node {\scriptsize\(\beta_{i-1}\)};
\draw (169:0.45) node {\scriptsize\(\alpha_i\)};
\draw (-3:0.45) node {\scriptsize\(\beta_i\)};
\draw (195:0.45) node {\scriptsize\(\alpha_{i+1}\)};
\draw[gray,very thick] (0,0) circle(1);
\end{tikzpicture}
\caption{An infinite leapfrog configuration as in Definition~\ref{d:leapfrog}.
Note that half of the angles between successive arcs in the configuration are on one side of the dashed curve \(\gamma\), and the other half are on the other side, so the arcs `leap' back and forth across \(\gamma\).
The leapfrog may be bounded on one side (if \(I=\ZZ_{\leq0}\) or \(I=\ZZ_{\geq0}\)), but not both.}
\label{f:leapfrog}
\end{figure}
As in \cite{CanFel}, a \emph{fan triangulation} of \((\disc,\pts)\) is one with no infinite leapfrog.
\end{defn}

This definition is illustrated in Figure~\ref{f:leapfrog}.
We note that there is heavy overlap between conditions \ref{d:leapfrog-alpha} and \ref{d:leapfrog-beta}; if \(I=\ZZ\) then they are even equivalent.
We have stated them separately to deal symmetrically with the cases \(I=\ZZ_{\leq0}\) and \(I=\ZZ_{\geq0}\).

The arcs of a triangulation \(\tri\) crossed by a fixed arc \(\gamma\) form a union of fans, as illustrated in Figure~\ref{f:fan-cross}.
If \(\tri\) is a fan triangulation then this union is finite, since otherwise the boundaries between the fans would form an infinite leapfrog configuration.
However, the fans involved may be finite or infinite.
\begin{figure}
\begin{tikzpicture}
\draw (90:2.5) edge[dashed,"\(\gamma\)",pos=0] (270:2.5);
\draw (-3,2) edge[very thick] (3,1);
\draw (3,1) edge[very thick] (-3,0);
\draw (-3,0) edge[very thick] (3,-1);
\draw (3,-1) edge[very thick] (-3,-2);
\foreach \x/\y in {3/2, 3/0, 3/-2, -3/1, -3/-1}
{\foreach \t in {-0.09,0,0.09}
{\draw (\x,\y+\t) node {\(\cdot\)};}}
\foreach \t in {1,2,3,4}
{\draw (-3,2) edge (3,1+0.2*\t);
\draw (3,1) edge (-3,2-0.2*\t);
\draw (3,1) edge (-3,0.2*\t);
\draw (-3,0) edge (3,1-0.2*\t);
\draw (-3,0) edge (3,-1+0.2*\t);
\draw (3,-1) edge (-3,-0.2*\t);
\draw (3,-1) edge (-3,-2+0.2*\t);
\draw (-3,-2) edge (3,-1-0.2*\t);}
\end{tikzpicture}
\caption{Crossings of an arc \(\gamma\) (dashed) with the arcs in a triangulation \(\tri\) (solid), yielding a collection of fans.
The bold arcs indicate the boundaries between different fans; if the number of such arcs were infinite, they would form an infinite leapfrog, and hence this cannot happen if \(\tri\) is a fan triangulation.
The fans shown are schematics only, and various kinds of convergence behaviour appear in practice.
For example, each fan may be finite or infinite, and the two outermost fans may be either bounded or unbounded.}
\label{f:fan-cross}
\end{figure}

The main result of this section is that the fan triangulations are precisely those corresponding to the cluster-tilting subcategories of \((\PYcat_n,\EE_{\rperp{\cat{D}}})\), that is, to the functorially finite weak cluster-tilting subcategories.
 One implication is relatively straightforward.

\begin{lem}
\label{l:ct-no-leapfrog}
If \(\cat{T}\) is a rigid and contravariantly finite subcategory of \((\PYcat_n,\EE_{\rperp{\cat{D}}})\), then \(\indec{\cat{T}}\) is a fan triangulation.
\end{lem}
\begin{proof}
Assume that \(\cat{T}\) is rigid, and \(\indec{\cat{T}}\) contains an infinite leapfrog configuration.
Since the set of marked points \(\close{\pts}_n\) is closed, we may choose the curve \(\gamma\) from Figure~\ref{f:leapfrog} to be an arc in \(\PYcat_n\).
By \cite[Prop.~3.14]{PaqYil} (see Proposition~\ref{p:PY-exts}), the \(\cat{T}\)-module \(\Hom_{\PYcat_n}(\cat{T},\susp\gamma)\) is supported on all the arcs in the leapfrog.
By the same result, we see that for there to exist a non-zero morphism between arcs in \(\cat{T}\), these arcs must share an endpoint: they cannot cross without contradicting Theorem~\ref{t:Ext-cross}, since \(\cat{T}\) is rigid.
Thus an arbitrary arc of \(\cat{T}\) has non-zero morphisms to at most three of the arcs in the leapfrog.
It follows that \(\Hom_{\PYcat_n}(\cat{T},\susp\gamma)\) is not finitely generated, and hence \(\cat{T}\) is not contravariantly finite.
\end{proof}

\begin{lem}
\label{l:limit-arcs}
Let \(p,q\in\close{\pts}_n\), and let \(\tri\) be a triangulation.
Assume that \(\tri\) contains a sequence of arcs \(\{\alpha_i\}_{i\in\NN}\), such that \(\alpha_i\) has endpoints \(p\) and \(x_i\), and the sequence \(x_i\) converges monotonically to \(q\).
Then either \(p=q\), the points \(p\) and \(q\) are adjacent, or the arc \(\alpha\) from \(p\) to \(q\) is contained in \(\tri\).
\end{lem}
\begin{proof}
Without loss of generality, we may assume that the sequence \(x_i\) converges to \(q\)  in the clockwise direction (against the orientation of the disc), so
\[p\leq q\leq\dotsb\leq x_2\leq x_1\leq p^+,\]
as illustrated in Figure~\ref{f:limit-arc}.
The case that \(p\leq x_i\leq q\leq p^+\) follows from this one by considering the opposite orientation of the disc.
\begin{figure}
\begin{minipage}{0.8\textwidth}
\begin{tikzpicture}[scale=2.4,every to/.style={hyperbolic disc}]
\draw (270:1) edge["\(\alpha\)", swap, dashed] (90:1);
\draw (90:1.1) node {\(q\)};
\foreach \t in {1,2,3,4,5}
{\draw (270:1) to (250-\t*20:1);
\draw (250-\t*20:1.1333) node {\(x_{\t}\)};}
\draw (270:1) to (130:1);
\draw (130:1.1333)  node {\(x_6\)};
\foreach \t in {-2.5,0,2.5}
{\draw (110+\t:0.9333) node {\(\cdot\)};}
\draw[gray,very thick] (0,0) circle(1);
\foreach \n in {1,2}
{ \draw (90-180*\n:1) circle(0.025) [fill=black];}
\draw (-90:1.1333) node {\(p\)};
\end{tikzpicture}\hfill
\begin{tikzpicture}[scale=2.4,every to/.style={hyperbolic disc}]
\draw (270:1) edge (90:1);
\draw (90:1.1) node {\(x_1\)};
\foreach \t in {3,4,5,6,7}
{\draw (270:1) to (90-\t*20:1);
\draw (90-\t*20:1.1333) node {\(x_{\t}\)};}
\draw (270:1) to (50:1);
\draw (50:1.1333) node {\(x_2\)};
\foreach \t in {-2.5,0,2.5}
{\draw (290+\t:0.9333) node {\(\cdot\)};}
\draw[gray,very thick] (0,0) circle(1);
\draw (-90:1) circle(0.025) [fill=black];
\draw (-90:1.1333) node {\(p=q\)};
\end{tikzpicture}
\end{minipage}
\caption{An illustration  of the assumptions of Lemma~\ref{l:limit-arcs}, in the cases \(p\ne q\) (left) and \(p=q\) (right).}
\label{f:limit-arc}
\end{figure}

Assume that \(p\ne q\), and these points are non-adjacent, so that there is an arc \(\alpha\) with endpoints \(p\) and \(q\).
Let \(\gamma\) be an arc, with endpoints \(x\) and \(y\), which crosses \(\alpha\), so that
\[p<x<q<y<p^+.\]
Since \(x_i\) converges to \(q\), we must have \(p<x<q\leq x_N<y<p^+\) for \(N\in\NN\) sufficiently large, and so \(\gamma\) crosses \(\alpha_N\).
Thus \(\gamma\notin\tri\).
It follows that \(\alpha\) does not cross any arc of the triangulation \(\tri\), and therefore \(\alpha\in\tri\).
\end{proof}

\begin{defn}
In the setting of Lemma~\ref{l:limit-arcs}, we say that the sequence \(\{\alpha_i\}_{i\in I}\) \emph{converges to an accumulation point} if \(p=q\), \emph{converges to a boundary segment} if \(p\) and \(q\) are adjacent, and \emph{converges to \(\alpha\)} otherwise. In the latter case, we call \(\alpha\) the \emph{limit} of the sequence \(\{\alpha_i\}_{i\in\NN}\), and write \(\alpha=\lim\alpha_i\).
\end{defn}

\begin{prop}
\label{p:fan-contra-fin}
If \(\tri=\indec{\cat{T}}\) is a fan triangulation, then \(\cat{T}\) is contravariantly finite.
\end{prop}
\begin{proof}
We abbreviate \(\HomT=\Hom_{\PYcat_n}(\cat{T},\blank)\colon\PYcat\to\Modcat{\cat{T}}\).
Let \(\gamma\in\PYcat_n\) be an arbitrary arc, and consider the \(\cat{T}\)-module \(\HomT\susp\gamma\); we aim to show that this module is finitely generated.
This is sufficient, since \(\susp\) permutes the indecomposable objects of the Krull--Schmidt category \(\PYcat_n\).

Since \(\tri\) is a fan triangulation, the support of \(\HomT\susp\gamma\) consists of a finite collection of fans; cf.~Figure~\ref{f:fan-cross}, replacing \(\gamma\) by \(\Sigma\gamma\).
Say a fan in this collection is \emph{clockwise bounded} if it has a maximal element when ordering its arcs clockwise around the common endpoint.
For example, if the arcs in Figure~\ref{f:limit-arc} (left), including \(\alpha\), form a fan in the support of \(\HomT\susp\gamma\), then this fan is clockwise bounded by \(\alpha\).

In a clockwise bounded fan, there is a non-zero map from any arc to the maximal element \(\alpha\).
As a result, if the fan is contained in the support of \(\HomT\susp\gamma\) then any non-zero element of \(\HomT\susp\gamma(\alpha)=\Hom_{\PYcat_n}(\alpha,\susp\gamma)\) generates a submodule \(N\) such that \(\HomT\susp\gamma/N\) is supported away from the chosen fan.
(Recall here that \(\HomT\susp\gamma\) is a contravariant functor on \(\cat{T}\).)
We conclude that \(\HomT\susp\gamma\) is finitely generated provided all the finitely many fans in its support are clockwise bounded.

Assume the support of \(\HomT\susp\gamma\) includes a fan which is not clockwise bounded. Then the support of \(\HomT\susp\gamma\) includes a sequence \(\{\alpha_i\}_{i\in\NN}\) as in Lemma~\ref{l:limit-arcs}, with the points \(x_i\) converging to \(q\) in the clockwise direction as in Figure~\ref{f:limit-arc}, but does not include the limit \(\alpha=\lim{\alpha_i}\) (if it exists).

By assumption, \(\Hom_{\PYcat_n}(\alpha_i,\Sigma\gamma)\ne0\) for all \(i\in\NN\).
First, assume that \(\gamma\) meets \(\alpha_i\) at \(x_i\) for some \(i\), and let \(v\) be the other endpoint of \(\gamma\).
For a non-zero morphism \(\alpha_{i-1}\to\susp\gamma\), we therefore must have
\[p<x_{i-1}=\susp x_i<\susp v<p^+.\]
But then
\[p<x_i<\susp x_i<\susp v<p^+,\]
so \(\susp\gamma\) and \(\alpha_i\) are disjoint, and there is no non-zero map from \(\alpha_i\) to \(\susp\gamma\); see Figure~\ref{f:gamma-vs-alpha_i} (left).
It follows that either \(\susp\gamma\) has one endpoint at \(p\), or \(\gamma\) crosses all the \(\alpha_i\).
We wish to show that, in either of these situations, the arc \(\alpha=\lim{\alpha_i}\) exists and satisfies \(\Hom_{\PYcat_n}(\alpha,\susp\gamma)\ne0\).
\begin{figure}
\begin{tikzpicture}[scale=2.1,every to/.style={hyperbolic disc}]
\draw (270:1) edge["\(\alpha\)", swap, dashed] (90:1);
\draw (90:1.1) node {\(q\)};
\draw (270:1) to (180:1);
\draw (180:1.2) node {\(x_{i-1}\)};
\draw (270:1) to (160:1);
\draw (160:1.1666) node {\(x_i\)};
\draw (270:1) to (140:1);
\draw (140:1.1666)  node {\(x_{i+1}\)};
\draw (235:1) to (160:1);
\draw (225:0.8) node {\(\gamma\)};
\draw (235:1.1) node {\(v\)};
\foreach \t in {-2.8,0,2.8}
{\draw (115+\t:0.9333) node {\({\cdot}\)};}
\draw[gray,very thick] (0,0) circle(1);
\foreach \n in {1,2}
{ \draw (90-180*\n:1) circle(0.025) [fill=black];}
\draw (-90:1.1333) node {\(p\)};
\end{tikzpicture}\hfill
\begin{tikzpicture}[scale=2.1,every to/.style={hyperbolic disc}]
\draw (270:1) edge["\(\alpha\)", swap] (90:1);
\draw (90:1.1) node {\(q\)};
\draw (270:1) to (180:1);
\draw (180:1.1333) node {\(x_{1}\)};
\draw (270:1) to (160:1);
\draw (160:1.1333) node {\(x_2\)};
\draw (270:1) to (140:1);
\draw (140:1.1333)  node {\(x_{3}\)};
\draw (235:1) to (90:1);
\draw (225:0.8) node {\(\gamma\)};
\draw (235:1.1) node {\(v\)};
\foreach \t in {-2.8,0,2.8}
{\draw (115+\t:0.9333) node {\({\cdot}\)};}
\draw[gray,very thick] (0,0) circle(1);
\foreach \n in {1,2}
{ \draw (90-180*\n:1) circle(0.025) [fill=black];}
\draw (-90:1.1333) node {\(p\)};
\end{tikzpicture}\hfill
\begin{tikzpicture}[scale=2.1,every to/.style={hyperbolic disc}]
\draw (270:1) edge["\(\alpha\)", near end] (90:1);
\draw (90:1.1) node {\(q\)};
\draw (270:1) to (180:1);
\draw (180:1.1333) node {\(x_{1}\)};
\draw (270:1) to (160:1);
\draw (160:1.1333) node {\(x_2\)};
\draw (270:1) to (140:1);
\draw (140:1.1333)  node {\(x_{3}\)};
\draw (270:1) to (45:1);
\draw (325:0.35) node {\(\gamma\)};
\draw (45:1.1) node {\(v\)};
\foreach \t in {-2.8,0,2.8}
{\draw (115+\t:0.9333) node {\({\cdot}\)};}
\draw[gray,very thick] (0,0) circle(1);
\foreach \n in {1,2}
{ \draw (90-180*\n:1) circle(0.025) [fill=black];}
\draw (-90:1.1333) node {\(p\)};
\end{tikzpicture}
\caption{Various ways in which the arc \(\gamma\) may interact with the sequence of arcs \(\alpha_i\).
In the left-most figure, there is no morphism \(\alpha_i\to\Sigma\gamma\).
In the other two, there are non-zero morphisms \(\alpha_i\to\Sigma\gamma\) for all \(i\), and also a non-zero morphism \(\alpha\to\Sigma\gamma\).}
\label{f:gamma-vs-alpha_i}
\end{figure}

If \(\gamma\) crosses all \(\alpha_i\), then \(\gamma\) has endpoints \(u\) and \(v\) such that
\[p<u< x_i<v< p^+.\]
However, \(\HomT\susp\gamma\) is not supported on \(\alpha\), and so \(\gamma\) and \(\alpha\) do not cross.
This means that \(\alpha_i\ne\alpha\) for any \(i\in\NN\), so \(q\) is an accumulation point, and
\[p\leq q\leq u< x_i<v<p^+.\]
Since the \(x_i\) converge to \(q\), it follows that \(u=q\), and hence \(p\ne q\) by the preceding inequality.
Since \(q\) is an accumulation point, it is not adjacent to \(p\), and so the arc \(\alpha\) between \(p\) and \(q\) exists and is in \(\tri\) by Lemma~\ref{l:limit-arcs}.
But now \(\alpha\) and \(\susp\gamma\) meet at the accumulation point \(q\) as in Proposition~\ref{p:PY-exts}(b), so that \(\Hom_{\PYcat_n}(\alpha,\susp\gamma)=\KK\) by this proposition; see Figure~\ref{f:gamma-vs-alpha_i} (centre).

Alternatively, the non-zero maps \(\alpha_i\to\susp\gamma\) for all \(i\) could arise from a meeting of \(\susp\gamma\) with these arcs at the common endpoint \(p\).
In this situation, we must have
\[p=u<v\leq x_i<p^+\]
for all \(i\), so that the non-zero morphism is from \(\alpha_i\) to \(\susp\gamma\), instead of from \(\susp\gamma\) to \(\alpha_i\); see Figure~\ref{f:gamma-vs-alpha_i} (right).
But the \(x_i\) converge to \(q\), so
\[p=u<v\leq q<p^+.\]
Thus, \(p\) and \(q\) are neither equal nor adjacent, so the arc \(\alpha\) between \(p\) and \(q\) is in \(\tri\) as before.
Moreover, there is a non-zero map \(\alpha\to\susp\gamma\) as required.
\end{proof}

\begin{thm}
\label{t:ct=fan}
The map \(\cat{T}\mapsto\indec{\cat{T}}\) is a bijection between cluster-tilting subcategories of \(\PYcat_n\) and fan triangulations of \((\disc,\close{\pts}_n)\).
\end{thm}
\begin{proof}
By Corollary~\ref{c:wct} it suffices to show that a weak cluster-tilting subcategory \(\cat{T}\) is functorially finite if and only if \(\tri=\indec{\cat{T}}\) is a fan triangulation.
If \(\cat{T}\) is a cluster-tilting subcategory, then it is in particular rigid and contravariantly finite, and so \(\tri\) is a fan triangulation by Lemma~\ref{l:ct-no-leapfrog}.

Conversely, if \(\tri\) is a fan triangulation, then \(\cat{T}\) is contravariantly finite by Theorem~\ref{p:fan-contra-fin}.
It remains to show that \(\cat{T}\) is covariantly finite, or equivalently that \(\cat{T}^\op\) is contravariantly finite in \(\PYcat_n^\op\).
Note that \(\indec{\cat{T}^\op}=\tri\), now viewed as a triangulation of the disc with opposite orientation.
But this is still a fan triangulation, so we conclude by Proposition~\ref{p:fan-contra-fin} again that \(\cat{T}^\op\) is contravariantly finite, and hence \(\cat{T}\) is covariantly finite, as required.
\end{proof}

\section{Mutability}
\label{s:mutation}

In triangulations of \((\disc,\pts)\) with \(\pts\) finite, every arc in a triangulation can be flipped, i.e.\ replaced with a different arc to obtain a new triangulation.
In the corresponding cluster categories, this operation appears as Iyama--Yoshino mutation of cluster-tilting subcategories \cite{IyaYos}.
However, when \(\pts\) is infinite, whether or not an arc can be flipped depends not just on the arc itself, but also on the triangulation being considered.
Similarly, when cluster-tilting subcategories are not additively finite, they may not always be mutated at an arbitrary indecomposable object.

\begin{defn}
\label{d:mutable}
Let \(\cat{C}\) be an extriangulated category with weak cluster-tilting subcategory \(\cat{T}\), and let \(\alpha\in\cat{T}\) be an indecomposable object.
We say \(\cat{T}\) is \emph{mutable at \(\alpha\)} if there is a unique indecomposable \(\alpha'\in\cat{C}\) such that \(\alpha'\) is not isomorphic to \(\alpha\) and the subcategory \(\cat{T}'=\add(\indec(\cat{T})\setminus\{\alpha\}\cup\{\alpha'\})\) is weak cluster-tilting.
\end{defn}

In the case that \(\cat{C}\) is the extriangulated category \((\PYcat_n,\EE_{\rperp{\cat{D}}})\), Definition~\ref{d:mutable} corresponds to \cite[Defn.~2.1]{BauGra} (for the case of the marked surface \((\disc,\close{\pts}_n)\)) under the bijections from Proposition~\ref{p:PY-bijection} and Corollary~\ref{c:wct}. That is, a weak cluster-tilting subcategory in \((\PYcat_n,\EE_{\rperp{\cat{D}}})\) is mutable at an indecomposable object \(\alpha\in\cat{T}\) if and only if the triangulation \(\indec{\cat{T}}\) is mutable at the arc \(\alpha\), in the sense of \cite[Defn.~2.1]{BauGra}.

Let \(\cat{T}\) be a rigid and additively closed subcategory of \((\PYcat_n,\EE_{\rperp{\cat{D}}})\).
Let \(\alpha\in\indec{\cat{T}}\) be an arc, with endpoints \(u,v\in\close{\pts}_n\).
We abbreviate \(\cat{T}\setminus\{\alpha\}=\add(\indec(\cat{T}\setminus\{\alpha\})\), and write
\begin{align*}
L_{\cat{T}}(\alpha,u)&=\{w\in\close{\pts}_n:\text{\(u<w<v<u^+\) and there is an arc between \(u\) and \(w\) in \(\cat{T}\)}\},\\
R_{\cat{T}}(\alpha,u)&=\{w\in\close{\pts}_n:\text{\(u<v<w<u^+\) and there is an arc between \(u\) and \(w\) in \(\cat{T}\)}\}.
\end{align*}
For an example, see Figure~\ref{f:apps} below, in which \(u^L\in L_{\cat{T}}(\alpha,u)\) and \(u^R\in R_{\cat{T}}(\alpha,u)\) (where \(\cat{T}\) is some rigid subcategory containing all of the arcs displayed in this figure).
Informally speaking, \(L_{\cat{T}}(\alpha,u)\) consists of the endpoints of arcs in \(\cat{T}\) obtained by rotating \(\alpha\) clockwise around \(u\), while \(R_{\cat{T}}(\alpha,u)\) consists of the endpoints of arcs in \(\cat{T}\) obtained by rotating \(\alpha\) anticlockwise around \(u\).

Since \(L_{\cat{T}}(\alpha,u)\) is a subset of the anticlockwise boundary segment from \(u\) to \(v\), and \(R_{\cat{T}}(\alpha,u)\) a subset of the anticlockwise boundary segment from \(v\) to \(u\), each inherits a total order from the cyclic order on \(\close{\pts}_n\).

\begin{lem}
\label{l:next-arc}
For \(\cat{T}\) and \(\alpha\) as above, if \(\alpha\) admits a left \((\cat{T}\setminus\{\alpha\})\)-approximation, then \(L_{\cat{T}}(\alpha,u)\) is either empty or has a maximum.
Dually, if \(\alpha\) admits a right \((\cat{T}\setminus\{\alpha\})\)-approximation, then \(R_{\cat{T}}(\alpha,u)\) is either empty or has a minimum.
\end{lem}

\begin{proof}
Assume \(\alpha\) admits a left \((\cat{T}\setminus\{\alpha\})\)-approximation \(\varphi\colon \alpha\to \ell\).
If \(L_{\cat{T}}(\alpha,u)\ne\varnothing\), then there is an arc \(\gamma\in\cat{T}\) between \(u\) and some point of \(L_{\cat{T}}(\alpha,u)\).
By Lemma~\ref{l:hourglass}, there is a non-zero morphism \(\psi\colon \alpha\to\gamma\).
Since \(\gamma\ne \alpha\) by construction, \(\psi\) factors over \(\ell\) (via \(\varphi\)).

By Lemma~\ref{l:1dim}, in fact \(\psi\) must factor over an indecomposable direct summand \(\lambda\) of \(\ell\).
By Lemma~\ref{l:hourglass} again, this happens if and only if \(\lambda\) is an arc between \(u\) and some \(w\in L_{\cat{T}}(\alpha,u)\).
Since \(\ell\) has finitely many indecomposable summands, we may therefore let \(w_0\) be the maximal element of \(L_{\cat{T}}(\alpha,u)\) such that the arc \(\lambda_0\) from \(u\) to \(w_0\) is an indecomposable summand of \(\ell\).

We claim that \(w_0\) is the desired maximal element of \(L_{\cat{T}}(\alpha,u)\).
By way of contradiction, let \(\beta\) be an arc of \(\cat{T}\) from \(u\) to \(w'\), such that \(u<w_0<w'<v<u^+\) (so that in particular \(w'\in L_{\cat{T}}(\alpha,u)\) is above \(w_0\) in the total ordering).
Then there is a non-zero morphism \(\xi\colon \alpha\to\beta\) which, since \(\beta\ne \alpha\), must factor over \(\varphi\).
But using Lemma~\ref{l:hourglass} again, we see that \(\xi\) does not factor over \(\lambda_0\), and hence cannot factor over any indecomposable summand of \(\ell\) without violating our assumption on \(w_0\).
Thus, we obtain a contradiction to Lemma~\ref{l:1dim}, from which we conclude that \(w_0\) is maximal in \(L_{\cat{T}}(\alpha,u)\).

The statement about \(R_{\cat{T}}(\alpha,u)\) follows, using the opposite orientation on \(\disc\).
\end{proof}

Applying Lemma~\ref{l:next-arc} to both endpoints of \(\alpha\in\cat{T}\), we see that if \(\alpha\) admits both left and right \(\cat{T}\setminus\{\alpha\}\)-approximations, then we are in the situation of Figure~\ref{f:apps}.
Here 
\[u^L=
\begin{cases}\susp u,&L_{\cat{T}}(\alpha,u)=\varnothing,\\
\max L_{\cat{T}}(\alpha,u),&\text{otherwise},
\end{cases}\quad
u^R=
\begin{cases}\susp^{-1} u,&R_{\cat{T}}(\alpha,u)=\varnothing,\\
\min R_{\cat{T}}(\alpha,u),&\text{otherwise},
\end{cases}\]
with analogous notation at \(v\).
\begin{figure}
\begin{tikzpicture}[scale=2.4,every to/.style={hyperbolic disc}]
\draw (135:1) to (60:1);
\draw (-45:1) to (30:1);
\draw (-45:1) to (-120:1);
\draw (-150:1) to (135:1);
\draw (135:1) edge node[above] {\(\alpha\)} (-45:1);
\draw[dotted] (-120:1) to (-150:1);
\draw[dotted] (30:1) to (60:1);
\draw (60:1.125) node {\(u^R\)};
\draw (-150:1.125) node {\(u^L\)};
\draw (-120:1.125) node {\(v^R\)};
\draw (30:1.15) node {\(v^L\)};
\draw (135:1.1) node {\(u\)};
\draw (-45:1.1) node {\(v\)};
\node at (45:0.62) {\(\beta'\)};
\node at (225:0.62) {\(\beta\)};
\node at (90:0.6) {\(\rho_1\)};
\node at (180:0.6) {\(\lambda_1\)};
\node at (270:0.6) {\(\rho_2\)};
\node at (0:0.6) {\(\lambda_2\)};
\draw [dashed,domain=297:335] plot ({cos(135)+0.5*cos(\x)}, {sin(135)+0.5*sin(\x)});
\draw [dashed,domain=117:155] plot ({cos(-45)+0.5*cos(\x)}, {sin(-45)+0.5*sin(\x)});
\draw[gray,very thick] (0,0) circle(1);
\end{tikzpicture}
\caption{Left and right \((\cat{T}\setminus\{\alpha\})\)-approximations to an arc \(\alpha\) in a rigid and additively closed subcategory \(\cat{T}\) of \((\PYcat_n,\EE_{\rperp{\cat{D}}})\).
The curves \(\lambda_i\) and \(\rho_i\) may be arcs or boundary segments, and no other arc of \(\cat{T}\) cuts the angles between \(\alpha\), \(\lambda_i\) and \(\rho_i\) at \(u\) and \(v\), by the definition of \(u^L\), \(u^R\), \(v^L\) and \(v^R\).
In particular, \(\add(\cat{T}\cup\{\beta,\beta'\})\) is rigid.}
\label{f:apps}
\end{figure}
We are now ready to prove the main theorem of the section.

\begin{thm}
\label{t:mutability}
Let \(\cat{T}\) be a weak cluster-tilting subcategory of \((\PYcat_n,\EE_{\rperp{\cat{D}}})\), let \(\tri=\indec{\cat{T}}\) be the corresponding triangulation (Corollary~\ref{c:wct}) and let \(\alpha\in\tri\) be an arc.
Then the following are equivalent:
\begin{enumerate}
\item\label{t:mutability-mutable} \(\cat{T}\) is mutable at \(\alpha\);
\item\label{t:mutability-diagonal} \(\alpha\) is the diagonal of a quadrilateral in \(\tri\);
\item\label{t:mutability-approx} \(\alpha\) admits left and right approximations by \(\cat{T}\setminus\{\alpha\}=\add(\indec(\cat{T})\setminus\{\alpha\})\).
\end{enumerate}
When these conditions are satisfied, mutating \(\cat{T}\) at \(\alpha\) corresponds, under the bijection from Corollary~\ref{c:wct}, to flipping the triangulation \(\tri\) at \(\alpha\).
\end{thm}
\begin{proof}
We first show that \ref{t:mutability-mutable} and \ref{t:mutability-diagonal} are equivalent.
By Corollary~\ref{c:wct}, \(\cat{T}\) is mutable at \(\alpha\) if and only if there is a unique triangulation \(\tri'\) such that \(\tri'=\tri\setminus\{\alpha\}\cup\{\alpha'\}\) for some arc \(\alpha'\ne\alpha\).
If \(\alpha\) is the diagonal of a quadrilateral in \(\tri\), then this is certainly the case, since we must take \(\alpha'\) to be the other diagonal of this quadrilateral: that is, \(\tri'\) is obtained by flipping \(\tri\) at \(\alpha\).

Conversely, if the arc \(\alpha'\) with the desired properties exists, then \(\alpha\) and \(\alpha'\) must cross, or \(\alpha\) would not cross any arc of \(\tri'=\tri\cap\tri'\cup\{\alpha'\}\), and \(\tri'\) would not be maximal.
Consider the quadrilateral \(Q\) whose corners are the endpoints of \(\alpha\) and \(\alpha'\), which must all be distinct since these edges cross.
By construction, the edges of \(Q\) do not cross either \(\alpha\) or \(\alpha'\).
Moreover, any arc \(\gamma\) crossing one of these edges must also cross either \(\alpha\) or \(\alpha'\).
Since \(\tri\) and \(\tri'\) are maximal collections of non-crossing arcs, it follows that no arc \(\gamma\in\tri\cap\tri'\) crosses any edge of \(Q\), and hence the edges of \(Q\) are arcs in \(\tri\) (or homotopic to the boundary).

Next we prove that \ref{t:mutability-diagonal} implies \ref{t:mutability-approx}.
Assume \(\alpha\) is contained in a quadrilateral of \(\cat{T}\), as shown in Figure~\ref{f:quad}.
Since \(\tri\) is a triangulation, no two arcs of \(\cat{T}\) cross, and so by Lemma~\ref{l:hourglass} the only morphisms between arcs of \(\cat{T}\) arise from meetings at a common endpoint.
Since any arc in \((\disc,\close{\pts}_n)\) cutting the angles between \(\alpha\), \(\lambda_i\) and \(\rho_i\) (indicated in Figure~\ref{f:quad} by the dashed curves) must cross one of these five curves, no such arc appears in \(\tri\).
Using Lemma~\ref{l:hourglass}, we see that every morphism \(\alpha\to \gamma\), with \(\gamma\in\tri\), factors over either \(\lambda_1\) or \(\lambda_2\), and so \(\alpha\) has a left \((\cat{T}\setminus\{\alpha\})\)-approximation \(\alpha\to\lambda_1\oplus\lambda_2\).
Similarly, \(\alpha\) has a right \((\cat{T}\setminus\{\alpha\})\)-approximation \(\rho_1\oplus \rho_2\to \alpha\).
\begin{figure}
\begin{tikzpicture}[scale=2.4,every to/.style={hyperbolic disc}]
\draw (135:1) to (45:1);
\draw (-45:1) to (45:1);
\draw (-45:1) to (-135:1);
\draw (-135:1) to (135:1);
\draw (135:1) edge node[right] {\(\alpha\)} (-45:1);
\node at (90:0.55) {\(\rho_1\)};
\node at (0:0.55) {\(\lambda_2\)};
\node at (-90:0.55) {\(\rho_2\)};
\node at (180:0.55) {\(\lambda_1\)};
\draw [dashed,domain=300:330] plot ({cos(135)+0.55*cos(\x)}, {sin(135)+0.55*sin(\x)});
\draw [dashed,domain=120:150] plot ({cos(-45)+0.55*cos(\x)}, {sin(-45)+0.55*sin(\x)});
\draw[gray,very thick] (0,0) circle(1);
\end{tikzpicture}
\caption{An arc \(\alpha\) contained in a quadrilateral, yielding left and right approximations.}
\label{f:quad}
\end{figure}

Finally, we show that \ref{t:mutability-approx} implies \ref{t:mutability-diagonal}.
Assuming \ref{t:mutability-approx}, it follows from Lemma~\ref{l:next-arc} that we are in the situation of Figure~\ref{f:apps}.
But then any arc \(\gamma\) of \((\disc,\close{\pts}_n)\) with one endpoint in the anticlockwise interval from \(u^L\) to \(v^R\) (inclusive) either has its other endpoint in the same interval, or crosses one of the other arcs in the figure, all of which are in \(\cat{T}\).
Thus, either \(u^R=v^L\) or the curve \(\beta\) between \(u^R\) and \(v^L\) is an arc of \(\tri\) (or a boundary segment, if \(\susp u^R=v^L\)) by maximality of \(\tri\).
In the latter case, \(\alpha\), \(\lambda_1\), \(\beta\) and \(\rho_2\) are all either boundary segments or arcs in \(\tri\), yet neither diagonal of the quadrilateral they bound may lie in \(\tri\) without violating the assumption on either \(\lambda_1\) or \(\rho_2\).
We conclude that \(u^L=v^R\).
Similarly \(u^R=v^L\), and \(\alpha\) is a diagonal of the quadrilateral with edges \(\lambda_i\) and \(\rho_i\) as in Figure~\ref{f:quad}.
\end{proof}

\begin{rem}
Note that Iyama--Yoshino \cite{IyaYos} give a direct proof of the statement `\ref{t:mutability-approx} implies \ref{t:mutability-mutable}' from Theorem~\ref{t:mutability} in the case that \(\cat{C}\) is a triangulated category.
This argument generalises in a straightforward way to the case that \(\cat{C}\) is a Frobenius extriangulated category (and \(\alpha\) is non-projective), but \((\PYcat_n,\EE_{\rperp{\cat{D}}})\) is not a category of this kind.

However, our combinatorial argument shows that when \(\alpha\in\cat{T}\) is mutable, there are exchange conflations
\[\alpha\infl\lambda_1\dsum\lambda_2\defl \alpha'\confl,\quad
\alpha'\infl \rho_1\dsum \rho_2\defl \alpha\confl\]
in \((\PYcat_n,\EE_{\rperp{\cat{D}}})\), exactly as in Iyama--Yoshino's situation.
Indeed, these are the triangles appearing in Theorem~\ref{t:ITcat}, resulting from the quadrilateral in Figure~\ref{f:quad}, which are also \(\EE_{\rperp{\cat{D}}}\)-conflations since \(\alpha\) and \(\alpha'\) cross.

In combinatorial terms, the equivalence of \ref{t:mutability-mutable} and \ref{t:mutability-diagonal} is also mentioned in \cite[Rem.~2.2]{BauGra}.
\end{rem}

\begin{eg}
\label{eg:fountain}
Let \(b\in\close{\pts}_n\).
Then the set \(\tri_b\) of all arcs in \((\disc,\close{\pts}_n)\) incident with \(b\) is a triangulation, which we call the \emph{fountain triangulation} with \emph{base} \(b\); an example is shown in Figure~\ref{f:fountain}.
Fountain triangulations are, in particular, fan triangulations, and so \(\tri_b\) corresponds to a cluster-tilting subcategory \(\cat{T}_b\) of \((\PYcat_n,\EE_{\rperp{\cat{D}}})\) by Theorem~\ref{t:ct=fan}.
\begin{figure}
\begin{tikzpicture}[scale=2.4,every to/.style={hyperbolic disc}]
\draw (270:1) to (90:1);
\foreach \t in {60,80,100,120,140}
{\draw (270:1) to (90+\t:1);
\draw (270:1) to (90-\t:1);}
\draw (270:1) to (130:1);
\draw (270:1) to (50:1);
\node at (-0.12,0.6) {\(\alpha\)};
\foreach \t in {-2.5,0,2.5}
{\draw (110+\t:0.933333) node {\(\cdot\)};
\draw (70+\t:0.933333) node {\(\cdot\)};}
\draw[gray,very thick] (0,0) circle(1);
\draw (270:1.083333) node {\(b\)};
\draw (90:1.083333) node {\(v\)};
\draw (90:1) circle(0.025) [fill=black];
\end{tikzpicture}
\caption{A fountain triangulation \(\tri_b\), consisting of all arcs in \((\disc,\close{\pts}_n)\) incident with the base \(b\).
If \(\alpha\in\tri_b\) is the arc between \(b\) and an accumulation point \(v\), then it is not mutable, since it does not lie in a quadrilateral.}
\label{f:fountain}
\end{figure}

Let \(\alpha\in\tri_b\), with endpoints \(b\) and \(v\).
If \(v\) is an accumulation point, as in Figure~\ref{f:fountain}, then \(\alpha\) is not mutable, since it does not lie in a quadrilateral in \(\tri_b\).
From Theorem~\ref{t:mutability}, it follows that \(\alpha\) does not admit left and right \((\cat{T}_b\setminus\{\alpha\})\)-approximations.
This also follows directly from Lemma~\ref{l:next-arc}, since both \(L_{\cat{T}_b}(\alpha,b)\) and \(R_{\cat{T}_b}(\alpha,b)\) are non-empty, but \(L_{\cat{T}_b}(\alpha,b)\) has no maximum, and \(R_{\cat{T}_b}(\alpha,b)\) no minimum, because \(v\) is an accumulation point.

On the other hand, if \(v\) is not an accumulation point, then \(\susp^{-1}v\), \(v\) and \(\susp v\) are distinct, cyclically consecutive, points of \(\close{\pts}_n\).
Letting \(\alpha^\pm\) be the arcs connecting \(b\) to \(\susp^{\pm1}v\) (or the appropriate boundary segments, if these points are adjacent), we see that \(\alpha\) lies in the quadrilateral bounded by \(\alpha^+\), \(\alpha^-\), and the boundary segments from \(\susp^{-1}v\) to \(v\) and \(v\) to \(\susp v\).
Thus \(\alpha\) is mutable; mutating \(\tri_b\) at \(\alpha\) replaces \(\alpha\) by the arc \(\alpha'\) between \(\susp^{-1}v\) and \(\susp v\).
\end{eg}

\begin{rem}
In some contexts in which a combinatorial mutation operation is not universally defined, this is viewed as a bug, whereby the combinatorial model is not sophisticated enough to capture all the cluster-theoretic information.
For example, a plabic graph \cite{Postnikov-PosGrass} determines a cluster algebra in which only certain mutations in the cluster algebra---those at vertices corresponding to square tiles---can be realised by a mutation operation on the graph itself.
A similar phenomenon occurs for cluster algebras defined via double wiring diagrams \cite{FomZel-DBCs} and their variations.
However, one still uses the traditional construction of the cluster algebra, in which mutations are always allowed away from a fixed set of frozen vertices; that is, it is not possible for a cluster variable to be mutable in one cluster but not in another.
In cases where these cluster structures have been categorified \cite{GLS-PFVs,JKS,Pressland-Postnikov}, the cluster-tilting objects are also always mutable at every non-projective summand, in contrast to the combinatorial models.
(In the case of plabic graphs, this issue may also be resolved at the combinatorial level by replacing the plabic graph by a more general object called a positroid weave \cite[Rem.~3.12]{CLSBW}.)

In our context, however, we suggest that the fact that arcs may be mutable in some triangulations but not others is not a bug, but an intrinsic feature of infinite rank cluster theory.
As remarked in the introduction, an arc is not mutable in a given triangulation precisely when its lambda length is expressible as the limit of a sequence of lambda lengths of other arcs in the triangulation \cite[\S3.2]{CanFel}, i.e.\ when this lambda length is not analytically independent of the others from the triangulation, despite its algebraic independence.
Thus this phenomenon has a natural interpretation in Teichmüller theory, and it is reasonable to expect this to be captured in the associated cluster theory.

Theorem~\ref{t:mutability} demonstrates that this mutability phenomenon also arises naturally in the categorification of this infinite-rank cluster combinatorics.
We note that a similar phenomenon has been observed for mutations of non-compact silting objects \cite{ALSV}, since a module may be mutable as a summand of one silting object but not another.
\end{rem}

\appendix
\section{Pullback of extriangulated substructures}
\label{appendix}

In this appendix we give a more general version of Lemma~\ref{l:functor-construction}.
For brevity, we freely use standard definitions, notation and results on extriangulated categories from \cite{NakPal}.

Let \((\cat{E},\EE,\s)\) and \((\cat{E}',\EE',\s')\) be extriangulated categories.
An extriangulated functor \(\cat{E}\to\cat{E}'\) consists of an additive functor \(F\colon\cat{E}\to\cat{E}'\) and a natural transformation \(\Gamma\colon\EE(\blank,\blank)\to\EE'(F\blank,F\blank)\) such that \(F\s(\delta)=\s'(\Gamma\delta)\) \cite[Defn.~3.15]{BTHSS}.
The fact that \(\Gamma\) is a natural transformation means that for any \(X,Z\in\cat{E}\), any \(\delta\in\EE(Z,X)\), and any morphisms \(z\colon Z'\to Z\) and \(x\colon X\to X'\), we have
\[\Gamma(z^*\delta)=(Fz)^*\Gamma\delta,\quad
\Gamma(x_*\delta)=(Fx)_*\Gamma\delta.\]
By a standard abuse of notation, we abbreviate the extriangulated functor to \(F\colon\cat{E}\to\cat{E}'\), although \(\Gamma\) is part of the data.

\begin{defn}
Let \(F\colon\cat{E}\to\cat{E}'\) be an extriangulated functor, and let \(\EE''\leq\EE'\) be an extriangulated substructure on \(\cat{E}'\); that is, \(\EE''\) is a subfunctor of \(\EE'\) such that \((\cat{E}',\EE'',\s'|_{\EE''})\) is again an extriangulated category.
For \(X,Z\in\cat{E}\), define \(F^{-1}\EE''(Z,X)\leq\EE(Z,X)\) to be the set of extensions \(\delta\in\EE(Z,X)\) such that \(\Gamma\delta\in\EE''(FZ,FX)\).
\end{defn}

\begin{prop}
\label{p:extri-pullback}
Let \(F\colon\cat{E}\to\cat{E}'\) be an extriangulated functor, and let \(\EE''\leq\EE'\) be an extriangulated substructure on \(\cat{E}'\).
Then \(F^{-1}\EE''\) is an extriangulated substructure on \(\cat{E}\), such that
\begin{enumerate}
\item\label{extri-pb-infl} a morphism \(f\colon X\to Y\) in \(\cat{E}\) is an \(F^{-1}\EE''\)-inflation if and only if \(f\) is an \(\EE\)-inflation and \(Ff\) is an \(\EE''\)-inflation,
\item\label{extri-pb-defl} a morphism \(g\colon Y\to Z\) in \(\cat{E}\) is an \(F^{-1}\EE''\)-deflation if and only if \(g\) is an \(\EE\)-deflation and \(Fg\) is an \(\EE''\)-deflation, and
\item\label{extri-pb-functor} \(F\) and \(\Gamma\) define an extriangulated functor \((\cat{E},F^{-1}\EE'',\s|_{F^{-1}\EE''})\to(\cat{E}',\EE'',\s|_{\EE''})\), and \(F^{-1}\EE''\) is the maximal extriangulated substructure on \(\cat{E}\) with this property.
\end{enumerate}
\end{prop}
\begin{proof}
We first show that \(F^{-1}\EE''(\blank,\blank)\) is a subfunctor of \(\EE(\blank,\blank)\).
Let \(\delta\in\EE(Z,X)\) and \(z\colon Z'\to Z\).
Then \(\Gamma(z^*\delta)=(Fz)^*\Gamma\delta\) since \(\Gamma\) is a natural transformation.
Thus, if \(\delta\in F^{-1}\EE''(Z,X)\), we have \(\Gamma\delta\in\EE''(FZ,FX)\), and hence \((Fz)^*\Gamma\delta\in\EE''(FZ',FX)\) by functoriality.
We conclude that \(z^*\delta\in F^{-1}\EE''(Z',X)\), as required.
Functoriality in the second argument is shown similarly.

To prove that \(F^{-1}\EE''\) is an extriangulated substructure, it is enough to establish property \ref{extri-pb-infl}, concerning its inflations.
Indeed, it will then follow that the \(F^{-1}\EE''\)-inflations are closed under composition, since this is true of both \(\EE\)-inflations and \(\EE''\)-inflations, and so \(F^{-1}\EE''\) is an extriangulated substructure by \cite[Prop.~3.16]{HLN1}.

To prove \ref{extri-pb-infl}, let \(f\) be an \(\EE\)-inflation, so there is an \(\EE\)-conflation
\[\begin{tikzcd}
X\arrow[infl]{r}{f}&Y\arrow[defl]{r}{g}&Z\arrow[confl]{r}{\delta}&\phantom{}
\end{tikzcd}\]
in \(\cat{E}\).
Since \(F\) is an extriangulated functor, it follows that
\[\begin{tikzcd}
FX\arrow[infl]{r}{Ff}&FY\arrow[defl]{r}{Fg}&FZ\arrow[confl]{r}{\Gamma\delta}&\phantom{}
\end{tikzcd}\]
is an \(\EE'\)-conflation in \(\cat{E}'\).
If we also assume that \(Ff\) is an \(\EE''\)-inflation, then there is a second conflation
\[\begin{tikzcd}
FX\arrow[infl]{r}{Ff}&FY\arrow[defl]{r}{g'}&Z'\arrow[confl]{r}{\delta'}&\phantom{}
\end{tikzcd}\]
in \(\cat{E}'\), with \(\delta'\in\EE''(Z',FX)\leq\EE'(Z',FX)\).
But then by \cite[Rem.~3.10]{NakPal}, there is an isomorphism
\[\begin{tikzcd}
FX\arrow[infl]{r}{Ff}\arrow[equal]{d}&FY\arrow[defl]{r}{Fg}\arrow[equal]{d}&FZ\arrow{d}{z}\arrow[confl]{r}{\Gamma\delta}&\phantom{}\\
FX\arrow[infl]{r}{Ff}&FY\arrow[defl]{r}{g'}&Z'\arrow[confl]{r}{\delta'}&\phantom{}
\end{tikzcd}\]
of \(\EE'\)-conflations, from which it follows that \(\Gamma\delta=z^*\delta'\in\EE''(FZ,FX)\).
Thus, \(\delta\in F^{-1}\EE''(Z,X)\), and \(f\) is an \(F^{-1}\EE''\)-inflation.

Conversely, assume that \(f\) is an \(F^{-1}\EE''\)-inflation, so there is an \(F^{-1}\EE''\)-conflation
\[\s|_{F^{-1}\EE''}(\delta)=(\begin{tikzcd}
X\arrow[infl]{r}{f}&Y\arrow[defl]{r}{g}&Z\arrow[confl]{r}{\delta}&\phantom{}
\end{tikzcd})\]
in \(\EE\), with \(\delta\in F^{-1}\EE''(Z,X)\).
Since \(F^{-1}\EE''\leq\EE\) is a subfunctor, and the realisation map is restricted from \(\s\), this is also an \(\EE\)-conflation, and in particular \(f\) is also an \(\EE\)-inflation.
Since \(\delta\in F^{-1}\EE''(Z,X)\), we have \(\Gamma\delta\in\EE''(FZ,FX)\).
Thus, \(F\s(\delta)=\s'(\Gamma\delta)\) is an \(\EE''\)-conflation, from which it follows that \(Ff\) is an \(\EE''\)-inflation.
This establishes \ref{extri-pb-infl}, and \ref{extri-pb-defl} is dual.

Property \ref{extri-pb-functor} holds essentially by construction.
First note that \(F\s(\delta)=\s'(\Gamma\delta)\) for all \(\delta\in\EE(Z,X)\) since \(F\) is an extriangulated functor.
If \(\delta\in F^{-1}\EE''(Z,X)\) then, by definition, \(\Gamma\delta\in\EE''(FZ,FX)\), and this identity may be written as
\[F\s|_{F^{-1}\EE''}(\delta)=\s'|_{\EE''}(\Gamma\delta).\]
Thus, \(F\) and \(\Gamma\) define an extriangulated functor as claimed.
If \(\delta\in\EE(Z,X)\) is a conflation in some substructure on \(\cat{E}\) for which \(F\) and \(\Gamma\) define an extriangulated functor to \((\cat{E}',\EE'',\s'|_{\EE''})\), then in particular \(\Gamma\delta\in\EE''(FZ,FX)\), so \(\delta\in F^{-1}\EE''(Z,X)\).
This establishes maximality.
\end{proof}

\begin{rem}
\begin{enumerate}
\item As one would expect, Proposition~\ref{p:extri-pullback} has several precursors.
For exact categories, this construction is described in \cite[\S1.4]{DRSS} (see also \cite{Heller} and, in the case that \(\cat{E}'\) is split exact, \cite[\S4.3]{Buehler}).
For exact \(\infty\)-categories, see \cite[Prop.~5.8]{Christ-Fukaya}.
As mentioned in Section~\ref{s:substruct}, the special case given in Lemma~\ref{l:functor-construction} was observed in \cite[Eg.~3.5(3)]{BalSte} (see also \cite[Eg.~2.3(2)]{Beligiannis-RHA}).
\item To deduce Lemma~\ref{l:functor-construction}  from Proposition~\ref{p:extri-pullback}, we take \(\cat{E}=\cat{C}\) and \(\cat{E}'=\cat{D}\), with their triangulated structures, and take \(\EE''=0\) to be the split extriangulated structure on \(\cat{D}\).
In this context, \(F\) is a triangle functor, so there is a specified natural transformation \(\eta\colon F\susp_{\cat{C}}\to\susp_{\cat{D}}F\).
The natural transformation \(\Gamma\colon\Hom_{\cat{C}}(Z,\susp X)\to\Hom_{\cat{D}}(FZ,F\susp_{\cat{C}} X)\isoto\Hom_{\cat{D}}(FZ,\susp_{\cat{D}} FX)\) required to make \(F\) an extriangulated functor is induced from \(\eta\) \cite[Thm.~2.33]{BTS}.

\item Whenever the category \(\cat{E}\) from Proposition~\ref{p:extri-pullback} is triangulated, the statements \ref{extri-pb-infl} and \ref{extri-pb-defl} simplify, since in this case every morphism in \(\cat{E}\) is both an inflation and a deflation.

\item Given extriangulated categories \((\cat{E},\EE,\s)\) and \((\cat{E}',\EE',\s')\), an additive functor \(F\colon\cat{E}\to\cat{E}'\) and a natural transformation \(\Gamma\colon\EE(\blank,\blank)\to\EE'(F\blank,F\blank)\), it would be interesting to find an efficient set of additional assumptions needed to guarantee the existence of a maximal extriangulated substructure on \(\cat{E}\) making \(F\) (and \(\Gamma\)) an extriangulated functor.
Proposition~\ref{p:extri-pullback} demonstrates that this happens whenever \(\cat{E}'\) admits a finer extriangulated structure (that is, one with more conflations) for which \(F\) is extriangulated.
This issue may thus be related to the question of whether additive categories admit maximal extriangulated structures, as answered positively by Rump \cite{Rump-MaxExact} for exact structures, and by Chen \cite[Thm.~7.10]{Chen-Thesis} for exact dg structures (which are enhancements of extriangulated structures).
\end{enumerate}
\end{rem}

\defbibheading{bibliography}[\refname]{\section*{#1}}\printbibliography
\end{document}